\documentclass{siamart220329}
\usepackage[utf8]{inputenc}
\usepackage{amsmath,amssymb,a4wide,bbm,esdiff,mathtools}
\usepackage{lipsum}
\usepackage{amsfonts}
\usepackage{graphicx}
\usepackage{epstopdf}
\usepackage{algorithmic}
\usepackage{bigints}
\usepackage{mathrsfs}
\usepackage{tikz}
\usepackage{pgfplots}
\usepackage{color}
\usepackage{enumitem}
\usepackage{soul}
\usepackage{lipsum}
\usepackage{amsfonts}
\usepackage{graphicx}
\usepackage{epstopdf}
\usepackage{algorithmic}
\usepackage{bigints}
\usepackage{mathrsfs}
\usepackage{tikz}
\usepackage{pgfplots}
\usepackage{color}
\usepackage{enumitem}
\usepackage{soul}

\usepackage{stackengine}

\newcommand{\parval}{\lambda} 
\newcommand{\leftint}{a}
\newcommand{\rightint}{b}
\newcommand{\numT}{M_T} 
 
\newcommand{\inter}{\textrm{int}\;}
\newcommand{\plen}{\ell}

\newcommand{\dtuple}{\alpha}
\newcommand{\wmac}{h}

\newcommand{\dimR}{d}
\newcommand{\dvar}{\beta} 
\newcommand{\findex}{{i}} 
\newcommand{\dimindex}{{j}} 
\newcommand{\Tindex}{{m}} 
\newcommand{\Tindexd}{{\mathbf{m}}} 
\newcommand{\LipK}{K} 
\newcommand{\M}{M} 
\newcommand{\Midpt}{\widehat{m}}
\ifpdf
  \DeclareGraphicsExtensions{.eps,.pdf,.png,.jpg}
\else
  \DeclareGraphicsExtensions{.eps}
\fi

\usepackage{enumitem}
\setlist[enumerate]{leftmargin=.5in}
\setlist[itemize]{leftmargin=.5in}

\newsiamremark{remark}{Remark}
\newsiamremark{hypothesis}{Hypothesis}
\crefname{hypothesis}{Hypothesis}{Hypotheses}
\newsiamthm{claim}{Claim}

\headers{SGD Markov Chains with Separable Functions}{D. Shirokoff, and P. Zaleski}

\title{Convergence of Markov Chains for Constant Step-size Stochastic Gradient Descent with Separable Functions\thanks{Submitted to the editors \today.
}}

\author{David Shirokoff\thanks{Department of Mathematical Sciences, New Jersey Institute of Technology, Newark, NJ
  (\email{shirokof@njit.edu}).} % \url{http://www.imag.com/\string~ddoe/}
\and Philip Zaleski\thanks{Department of Mathematical Sciences, New Jersey Institute of Technology, Newark, NJ
  (\email{pz85@njit.edu}).}
}

\usepackage{amsopn}

\makeatletter
\newcommand*{\addFileDependency}[1]{% argument=file name and extension
  \typeout{(#1)}% latexmk will find this if $recorder=0 (however, in that case, it will ignore #1 if it is a .aux or .pdf file etc and it exists! if it doesn't exist, it will appear in the list of dependents regardless)
  \@addtofilelist{#1}% if you want it to appear in \listfiles, not really necessary and latexmk doesn't use this
  \IfFileExists{#1}{}{\typeout{No file #1.}}% latexmk will find this message if #1 doesn't exist (yet)
}
\makeatother

\newtheorem{example}{Example}

\pgfplotsset{compat=1.18}
\begin{document}
\maketitle

% REQUIRED
\begin{abstract}
Stochastic gradient descent (SGD) is a popular algorithm for minimizing objective functions that arise in machine learning.  For constant step-sized SGD, the iterates form a Markov chain on a general state space.  Focusing on a class of separable (non-convex) objective functions, we establish a ``Doeblin-type decomposition,'' in that the state space decomposes into a uniformly transient set and a disjoint union of absorbing sets.  Each of the absorbing sets contains a unique invariant measure, with the set of all invariant measures being the convex hull. Moreover the set of invariant measures are shown to be global attractors to the Markov chain with a geometric convergence rate.  The theory is highlighted with examples that show: (1) the failure of the diffusion approximation to characterize the long-time dynamics of SGD; (2) the global minimum of an objective function may lie outside the support of the invariant measures (i.e., even if initialized at the global minimum, SGD iterates will leave); and (3) bifurcations may enable the SGD iterates to transition between two local minima.  Key ingredients in the theory involve viewing the SGD dynamics as a monotone iterated function system and establishing a ``splitting condition'' of Dubins and Freedman 1966 and Bhattacharya and Lee 1988.        
\end{abstract}

% REQUIRED
\begin{keywords}
    Stochastic gradient descent, Diffusion approximation, Doeblin-type decomposition, Markov chains, Spectral gap, Constant step-size, Bifurcations, Iterated function systems
\end{keywords}

% REQUIRED
\begin{AMS}
  68W20, 68W40, 37A30, 60J20
\end{AMS}

 %==============================================================================
\section{Introduction}\label{Sec:Introduction}
%==============================================================================
In recent years, stochastic gradient descent (SGD) \cite{RobbinsMonro1951} has become an immensely popular algorithm for minimizing objective functions $F: \mathbb{R}^\dimR \xrightarrow{} \mathbb{R}$ of the form 
\begin{equation}
\label{F}
    F(x)= \frac{1}{n} \sum_{\findex=1}^n f_\findex(x) \, , \qquad \textrm{where} \qquad 
    f_\findex : \mathbb{R}^\dimR \xrightarrow{} \mathbb{R} \, \qquad (f_\findex \neq 0)\, .
\end{equation}
Define the maps $\varphi_\findex : \mathbb{R}^\dimR \xrightarrow{} \mathbb{R}^\dimR$ as
\begin{equation}\label{Eq:SGD_Maps}
    \varphi_\findex(x) :=x - \eta \nabla f_{\findex} (x) \qquad (1 \leq \findex \leq n) \, ,
\end{equation}
where the parameter $\eta > 0$ is the \emph{step-size} (or \emph{learning rate}). 

In its simplest form, the constant step-size SGD generates a random sequence of iterates 
$\{X_0, X_1, X_2, \ldots \}$ via the following dynamics:

Given $X_0 \in \mathbb{R}^\dimR$, sampled from an initial distribution $\mu_0$, update
\begin{align}\label{Eq:SGDIterates}    
    X_{k+1} &= \varphi_{\findex_k}(X_k)     \, \qquad \textrm{where} \qquad \findex_k \in \{ 1,2, 3, \ldots, n\} =: [n] 
\end{align}
is drawn independently and identically from a uniform distribution (i.i.d.); each map $\varphi_i$ has probability $1/n$ of being applied in \eqref{Eq:SGDIterates}. 

Since the values of $\findex_k$ are i.i.d., the random variable $X_{k+1}$ satisfies the Markov property and the sequence $\{X_k\}_{k \geq 0}$ forms a Markov chain taking on values in the \emph{general state-space} $\mathbb{R}^\dimR$ (see \cite{Hairer2021, Meyn_Tweedie} for background on Markov chains in general state spaces). 

Associated to each $X_k$ is the corresponding law of probability $\mu_k \in \mathscr{P}(\mathbb{R}^{\dimR})$ characterizing the probability distribution of $X_k$
\begin{align*}
    \mu_k(A)  := \textrm{Prob}(X_k \in A)  \qquad (A \in \mathcal{B}(\mathbb{R}^\dimR)) \, . 
\end{align*}
Here $\mathscr{P}(\mathbb{R}^{\dimR})$ is the space of probability measures over the Borel $\sigma$-algebra $\mathcal{B}(\mathbb{R}^\dimR)$.

The probability laws for a Markov chain then evolve via deterministic linear dynamics according to a Markov operator 
\begin{align}\label{Eq:MarkovOperatorDynamics}
    \mu_{k+1} = \mathcal{P} \mu_k \, .
\end{align}
For general state spaces, the Markov operator is defined as
\begin{equation}\label{Eq:KernelEvolution}
    (\mathcal{P} \mu)(A) := \int_{\mathbb{R}^\dimR} p(x, A ) \, d \mu(x)   \qquad (A \in \mathcal{B}(\mathbb{R}^\dimR))
\end{equation}
where $p: \mathbb{R}^\dimR \times \mathcal{B} (\mathbb{R}^\dimR) \xrightarrow{} [0,1]$ is the \emph{transitional kernel}
\begin{align*}
    p(x,A) := \textrm{Prob}\left( X_{k+1} \in A  \; | \;  X_k = x \right) \, . 
\end{align*}    
For each $x \in \mathbb{R}^{\dimR}$, $p(x, \cdot)$ is a probability measure, while for each Borel set $A$, $p(\cdot, A)$ is a measureable function so that  $\mathcal{P}$ is well defined to act on probability measures (or more generally finite measures).

Intuitively, the value of $p(x,A)$ measures the probability that the Markov chain transitions from a point $x$ into the set $A$ (which is the infinite dimensional analogue to matrix elements of a Markov matrix).  Hence, for the SGD Markov chain \eqref{Eq:SGDIterates}, $p$ is the fraction of maps $\varphi_\findex$ that map $x$ into $A$
\begin{equation}\label{Def:TransitionKernel}
    p(x,A) = \frac{1}{n} \sum_{\findex=1}^n \chi_{A}\big( \varphi_\findex(x) \big) \, ,    
\end{equation}
where $\chi_A$ is the characteristic function of the set $A$. 

For SGD \eqref{Eq:SGDIterates}, $\mathcal{P}$ then takes the form
\begin{align}\label{Eq:SGD_Markov_Operators}
    (  \mathcal{P} \mu)(A) = \frac{1}{n} \sum_{\findex= 1}^n \mu\left( \varphi_\findex^{-1}(A) \right) \, , 
\end{align}
where $\varphi_\findex(x)$ are defined in \eqref{Eq:SGD_Maps} and $\varphi_\findex^{-1}(A)$ is the preimage of $A$. 

A probability measure $\mu^\star$ is invariant (or stationary) with respect to the operator $\mathcal{P}$ if 
\begin{equation}
      \mathcal{P} \mu^\star = \mu^\star \, .
\end{equation}

The purpose of this work is to establish the convergence of the probability measures $\mu_k$ for constant step-size SGD.  We restrict our attention to separable (but non-convex) objective functions $F$ and $f_\findex$. This has the advantage of enabling a simple, yet relatively complete, characterization of the how the probability measures $\mu_k$ converge to a convex combination of the extreme points of the invariant measures.  Our results make use of techniques from the theory of \emph{iterated function systems} for monotone maps. A technical part of the proofs involves verifying a splitting condition \cite{Dubins66, Bhattacharya88} associated to Markov operators for these maps.

Analyzing the exact dynamics defined by the operator $\mathcal{P}$ then enables a rigorous study of bifurcations.  In particular, we provide an exact bifurcation study of whether iterates of SGD escape a local minima of $F$ --- contradicting predictions inferred by the diffusion approximation.

%==============================================================================
\subsection{Background on SGD with Vanishing \texorpdfstring{$\eta_k \rightarrow 0$}{Lg} 
 } 
%==============================================================================
An intuitive motivation for the dynamics \eqref{Eq:SGDIterates} is that the expectation of $X_{k+1}$ is a gradient step of $F$ evaluated at $X_k$, i.e., 
{\begin{align}\label{Eq:ExpStep}
    \mathbb{E} (X_{k+1}|X_k) = X_k - \eta \nabla F(X_k) \, .    
\end{align}}
Equation \eqref{Eq:ExpStep} implies that on average, one expects the iterates $X_k$ to move towards local minima of $F(x)$.  This observation can be made precise and has lead to a significant body of work on SGD in the small $\eta \ll 1$ step-size limit.  When the step-size $\eta$ in \eqref{Eq:SGDIterates} is {vanishing}, that is, $\eta$ is replaced with $\eta_k$ with $\eta_k \rightarrow 0$ as $k \rightarrow \infty$, it is well-established that the iterates $X_k$ converge the dynamics of the ordinary differential equation 
\begin{align}\label{Eq:ODE}
    \dot{x} = -\nabla F(x) \, . 
\end{align}
For instance, see \cite{Duflo1996}, \cite[Proposition 4.1--4.2]{Benaim2006} and \cite[Chapter 2]{Borkar2008} (and references within). Closely related to the asymptotic trajectory \eqref{Eq:ODE} are a range of results establishing that $X_k$ converges (almost surely) to minimizers of $F$, e.g. \cite{Wojtowytsch2023a, DereichKassing2024} for functions $F$ satisfying a Łojasiewicz inequality in lieu of convexity.

%\L ojasiewicz can also be used for a Polish accent L

%==============================================================================
\subsection{Background for SGD with Constant \texorpdfstring{$\eta$}{Lg}}\label{Subsec:SGD_Background_ConstEta}
%==============================================================================
While much is known about SGD with {vanishing} step-sizes {($\eta_k \rightarrow 0$)}, {there are still many open questions} %far less is known 
in the constant step-size setting when $F$ is non-convex.  In particular, regarding \eqref{Eq:SGDIterates}: Does the Markov chain $X_k$ remain trapped in an energy well of $F$? Or explore all minima of $F$? And if so, over what time-scales?  Answers to these questions may help to provide insight into the {early phase of SGD with vanishing step-sizes.} 

The difficulty in establishing a general theory for either the Markov chain $X_k$, or the associated probability laws $\mu_k$, is highlighted by the generality of SGD.  For instance, the dynamics \eqref{Eq:SGDIterates} include, as special cases:
\begin{itemize}
    \item All continuous $1$-dimensional deterministic iterative maps. This includes both the Logistic map and Tent map. For instance, take $n=1$ and $F(x) = -\frac{3}{2}x^2 + \frac{4}{3}x^3$ so that $X_{k+1} = 4 X_k (1- X_k)$ when $\eta = 1$. Varying $\eta \in (0, 1]$, the iterates $X_k$ exhibit the classic behavior of period doubling and the emergence of chaos;
    \item Random walks, e.g., set $F(x) = 0$ with $f_1(x) = x$ and $f_2(x) = -x$;
    \item Infinite Bernoulli convolutions, which up to a linear change of variables, have the form $F(x) = x^2$, with $f_1(x) = (x-1)^2$, $f_2(x) = (x + 1)^2$.  Erd\H{o}s \cite{Erdos1939} showed that in this quadratic setting the corresponding invariant measures may be singular.  Quadratic models have also found recent applications in biological settings \cite{CountermanLawley2021}, and remain an area of research in dynamical systems \cite{Feng2003, Benjamini2009, Jordan2011, Kempton2015, Bandt2018, Batsis2021}.    
    \item {Deterministic gradient descent with constant step-size. This includes settings where the objective function may have multiples length scales \cite{KongTao2020}: If the smallest length scale is under-resolved (i.e., much smaller than $\eta$), the dynamics can be approximated as gradient descent with stochastic forcing and exhibit chaos.}
\end{itemize}
In light of the (extremely) broad class of dynamics given by \eqref{Eq:SGDIterates}, we focus on separable $f_\findex$.

Continuous time partial differential equations (PDEs) provide one approach for approximating the evolution \eqref{Eq:MarkovOperatorDynamics}. Treating $\eta \ll 1$ as a small parameter, \eqref{Eq:ODE} can be viewed as a leading order approximation to the dynamics for $X_k$; this also yields a corresponding advection PDE approximation for \eqref{Eq:MarkovOperatorDynamics}.  Truncating formal expansions of the operator $\mathcal{P}$ at progressively higher orders in the asymptotic parameter $\eta$ \cite{LiTaiE2017} then yield successive improvements. For instance, a second order expansions in $\eta$ to $\mathcal{P}$ yields a \emph{diffusion approximation} to \eqref{Eq:MarkovOperatorDynamics} (see \S\ref{sec:DiffApprox}). The diffusion approximation is known to accurately describes the evolution of $\mu_k$ for finite time \cite{LiTaiE2017, HuLiLiLiu2017, FengLiLiu2018}, and infinite time when $F$ is convex \cite{FengGaoLiLiuLu2020}.  The diffusion approximation has also been used to gain insight into SGD dynamics and (in the regime for which it is valid) can estimate the probability distribution of $X_k$ near the minimum of $F$ (cf. \cite{NIPS2014_f29c21d4} for estimates and convergence rates when $f_\findex$ are convex).  As we discuss below, the diffusion approximation can fail significantly to capture the correct long-time dynamics such as the number of, and regularity of, invariant measures (cf. \cite{McCann2021}). 

Variants of the diffusion equation have also been used to model SGD dynamics \cite{mandt2015continuous,chaudhari2018deep}. Rigorous PDE theory, such as the existence and geometric convergence to the invariant measure, have also been established for diffusion equation models \cite{Wojtowytsch2023b}.  Partial differential equation models can also arise from SGD as limiting dynamics where the asymptotic parameter is the dimension (not $\eta$!) \cite{ArousGheissariJagannath2022}. 

One approach to establish convergence of $\mathcal{P}^k \mu$ to a unique invariant measure (on a state space $X$), is through a Doeblin condition of the form
\begin{align}\label{Eq:Doeblin}
    p(x, A) \geq \epsilon \nu(A) \qquad \qquad \textrm{for all} \; A \in \mathscr{B}(X), \; x \in X
\end{align}
holds for some $\nu \in \mathscr{P}(X)$, $\epsilon > 0$. %and all $A \in \mathscr{B}(X)$, $x \in X$.  
Some algorithmic variants of SGD, such as those that add random noise at each step (e.g., \emph{stochastic gradient Langevin dynamics}), or make strong assumptions on $f_i$ which mimic random noise (e.g., \cite{NEURIPS2021_21ce6891}), closely resemble a stochastic ODE. In these cases, conditions such as \eqref{Eq:Doeblin} may be verified to prove convergence to a unique invariant measure.  The version of SGD \eqref{Eq:SGDIterates}, in general, does not satisfy \eqref{Eq:Doeblin} (even when $X$ is restricted to the absorbing sets).

%Closely related to SGD are algorithmic variants that add random noise at each step such as \emph{stochastic gradient Langevin dynamics}.  The structure of the noise in these models more closely resembles a stochastic ODE; subsequently much more is known regarding the convergence and invariant measures as tools from SDE theory are more readily applicable. For instance, the Markov chain probability laws in models with random noise converge to a unique invariant measure, e.g., \cite{NEURIPS2021_21ce6891}.  In contrast, the SGD dynamics considered in \eqref{Eq:SGDIterates} do not add additional random noise. 

Our main result establishes the convergence of the exact Markov chain dynamics \eqref{Eq:MarkovOperatorDynamics}. We emphasize that we do not make use of concepts that often arise in Markov chains with random noise, e.g., $\varphi$-irreducibility (which is a notion of irreducibility for general state spaces), detailed balance/reversibility, or \eqref{Eq:Doeblin}. We also make no continuous time approximation. 
 Rather, the starting point is to view \eqref{Eq:SGDIterates} as a random \emph{iterated function system} (IFS).  We then show that the SGD dynamics \eqref{Eq:SGDIterates} satisfy the \emph{splitting} conditions \cite{Dubins66, Bhattacharya88} that guarantees convergence to an invariant measure for IFS with monotone maps.  

%==============================================================================
\subsection{Background on Iterated Function Systems}
%==============================================================================
There is a long history of iterated function systems theory in dynamical systems --- notably as means to construct fractals.  Two broad approaches to establish existence and convergence to invariant measures in IFS are: (i) to show that the maps $\varphi_\findex$ satisfy a contraction condition, or (ii), which is the approach we adopt, to show that the maps are monotone and satisfy a splitting condition.  For example, when $\varphi_\findex$ are contractions (which can occur when $f_\findex$ is convex and $\eta$ is constant but small enough), a result by Hutchinson \cite[Section 4.4, Theorem 1]{Hutchinson81} shows that $\mathcal{P}$ is also a contraction on the space of probabilities and $\mu_k$ converges geometrically to a unique invariant measure.  

Convergence results for IFS have also been generalized to allow for weaker notions of contractivity \cite{Diaconis99, Gelfert2021, Herkenrath2007, Myjak2003, Steinsaltz99, Stenflo2012, Wu2004}. Contraction conditions have recently been used to establish convergence in specific instances (e.g., when $f_\findex$ are convex) of SGD \cite{GuptaHaskell2021, GuptaChenPiTendolkar2020}.  Unfortunately, contractivity conditions can be difficult to verify (or possibly fail) in practical settings. For instance, when $f_\findex$ is non-convex, the map $\varphi_\findex$ (for small $\eta$) is never contractive in the standard Euclidean norm.  Since we are particularly interested  in settings where $f_\findex$ and $F$ are non-convex we avoid the use of contractions and instead make use of monotone maps \cite{Dubins66, Bhattacharya88,Yahav1975,Bhattacharya99}.

%==============================================================================
\subsection{Contributions and Organization of the Paper} 
%==============================================================================
For constant step-size SGD with separable (non-convex) functions, the main contributions of the paper are: 
\begin{itemize}
    \item A decomposition of the Markov chain state space for \eqref{Eq:MarkovOperatorDynamics} into a uniformly transient set and disjoint union of absorbing sets --- each supporting a unique invariant measure (these are the extreme points in the set of invariant measures);
    %A decomposition of the Markov chain state space for \eqref{Eq:MarkovOperatorDynamics} into a uniformly transient set and disjoint union of absorbing sets. Each absorbing set contains a unique invariant measure, which is an extreme point in the set of invariant measures;    
    \item The set of invariant measures are attractors with $\mathcal{P}^k\mu$ converging  geometrically to a convex combination of the extreme points; 
    \item Bounds on the number and support of the invariant measures;
    \item Examples demonstrating the failure of the diffusion approximation; that the invariant measures may be supported outside a neighborhood of the global minima of $F$; and a rigorous bifurcation study of the dynamics \eqref{Eq:MarkovOperatorDynamics}.
\end{itemize}
The manuscript starts in \S\ref{Sec:MainResult} by introducing the assumptions on $f_i$ and presenting the main result. Examples are then provided in \S\ref{Sec:Examples1d}.  The remainder of the paper is devoted to establishing the main result: \S\ref{Sec:MathBackground} introduces the required mathematical notation while \S\ref{Sec:MarkovCvg} reviews convergence theory for Markov operators with monotone maps $\varphi_i$. Several theoretical results for SGD in one dimension are established in \S\ref{Sec:Lemmas1d_SGD}, which are then expanded on in a non-trivial way to prove the main result in \S\ref{Sec:ProofsMainResult}.

%applied to  in \S\ref{Sec:Examples1d} to carry out a bifurcation analysis of \eqref{Eq:MarkovOperatorDynamics} for non-convex $F$. The examples show the failure of the diffusion approximation, and that the invariant measures may be supported outside a neighborhood of the global minima of $F$.  

%==============================================================================
\section{Main Result}\label{Sec:MainResult} 
%==============================================================================
In this section we introduce the assumptions and main result.

%==============================================================================
\subsection{Assumptions and Problem Setting}
\label{subsec:Assumptions}
%==============================================================================
Throughout, we assume that each of the functions $f_\findex$ (and consequently $F$) are \emph{separable}, that is, they are the sum of single variable functions $f_\findex^{(\dimindex)} : \mathbb{R} \rightarrow \mathbb{R}$ 
\begin{align}\label{Def:Separable}
    f_\findex(x) = \sum_{\dimindex=1}^{\dimR} f_{\findex}^{(\dimindex)}(x_\dimindex) 
    \qquad \textrm{where} \qquad 
    x = (x_1, \, x_2, \, \ldots, \, x_{\dimR} ) \in \mathbb{R}^{\dimR} \, .
\end{align}
With this convention, $f_\findex^{(\dimindex)}$ is allowed to be $0$, however to avoid trivialities $f_\findex \neq 0$. 

For functions of the form \eqref{Def:Separable}, the maps $\varphi_\findex: \mathbb{R}^d \rightarrow \mathbb{R}^d$ become
\begin{equation}\label{Eq:SepVarphi}
   \varphi_\findex(x) = \begin{pmatrix}  \varphi_\findex^{(1)}(x_1), &  \varphi_\findex^{(2)}(x_2), & \ldots\;, & \varphi_\findex^{(d)} (x_d)
    \end{pmatrix}     
\end{equation}
where each of the components of $\varphi_\findex$ is a single variable function of $x_\dimindex$
\begin{align*}    
    \varphi_\findex^{(\dimindex)}(s) = s - \eta \frac{d}{ds}f_\findex^{(\dimindex)}(s) \qquad (1 \leq \dimindex \leq \dimR) \, . 
\end{align*}

For all non-zero $f_\findex^{(\dimindex)}$ $(\findex = 1, \ldots, n$;  $\dimindex = 1, \ldots, \dimR)$ we assume that
\begin{enumerate}[itemsep=5pt, topsep=5pt, label=(A{\arabic*})]
    \item $f_\findex^{(\dimindex)}$ is continuously differentiable, i.e., $f_\findex^{(\dimindex)} \in C^{1}(\mathbb{R})$. \label{A1}     
    \item $f_\findex^{(\dimindex)}$ has a finite number of critical points. \label{A2} 
    \item $f_\findex^{(\dimindex)}(x) \rightarrow \infty $ as $|x|\rightarrow \infty$. \label{A3}
\end{enumerate}
For each family of functions $\{f_\findex^{(\dimindex)}\}_{\findex=1}^{n}$ the set of critical points 
\begin{align*}
    \mathcal{C}^{(\dimindex)} = \bigcup_{\findex=1}^n \left\{ x \in \mathbb{R} \; : \; \frac{d}{dx}f_\findex^{(\dimindex)}(x) = 0, \; f_\findex^{(\dimindex)} \neq 0 \right\}   
\end{align*}
is then finite and contained in an interval of the form
\begin{align}
\label{I_state_space}
    I^{(\dimindex)} := [ \leftint, \rightint ] \, ,
\end{align}
where we define $\leftint$ and $\rightint$ to be the smallest and largest elements of $\mathcal{C}^{(\dimindex)}$.  With this notation, the general \emph{state space} will be
\begin{align}\label{Def:Statespace}
    I := I^{(1)} \times I^{(2)} \times \ldots \times I^{(d)} \subset \mathbb{R}^{\dimR} \,.
\end{align}
We further assume that
\begin{enumerate}[itemsep=5pt, topsep=5pt, label=(A{\arabic*})]
    \setcounter{enumi}{3}
    \item\label{A4} The derivatives of $f_\findex^{(\dimindex)}$ are $\LipK$-Lipschitz on $I^{(\dimindex)}$, i.e., for some $K > 0$
    \begin{align*}
        \left|\dfrac{d}{dx}f_\findex^{(\dimindex)}(x) - \dfrac{d}{dy}f_\findex^{(\dimindex)}(y) \right| \leq \LipK \,  |x - y| \qquad \textrm{for} \quad x,y \in I^{(\dimindex)}\; \textrm{and} \; 1 \leq \findex \leq n \, . 
    \end{align*}    
    \item\label{A5} (\emph{Inconsistent optimization}) For each $1 \leq \dimindex \leq \dimR$ the functions $f_\findex^{(\dimindex)}$ share no common critical point, i.e., 
        \begin{align*}
            \bigcap_{\findex = 1}^n \left\{x_\dimindex \in \mathbb{R} \: : \;\frac{d}{dx}f_\findex^{(\dimindex)}(x_\dimindex) = 0 \right\} = \phi \, . 
        \end{align*}
\end{enumerate}

The assumptions \ref{A1}--\ref{A4} (or minor variations of) are standard in the gradient descent optimization literature. Together they ensure that when $\eta < \LipK^{-1}$ each $\varphi_\findex^{(\dimindex)}$ is an increasing function --- a property we use in the proofs.  Condition \ref{A2} is made for simplicity to rule out complications when $f_\findex$ admits an infinite number of critical points, however the condition can be relaxed, for instance some of the theory we present generalizes to allow for some $f_\findex$ to be constant on an interval.

Assumption \ref{A5} is sometimes referred to as \emph{inconsistent optimization} since it implies the $f_\findex$ do not share a common minimizer.  Condition \ref{A5} also implies that for each $\dimindex$ at least two functions in the set $\{ f_\findex^{(\dimindex)} \}_{\findex=1}^{n}$ are non-zero (otherwise \ref{A5} fails trivially).

While it may appear that \ref{A5} is overly simplifying, it is necessary to establish both convergence in a (strong) metric, and  uniform geometric convergence rates in \cref{Main_thm_basin}.  When Assumption \ref{A5} is removed, the convergence theory we build on from \cite{Bhattacharya88, Dubins66} no longer applies as stated.

The necessity of Assumption \ref{A5} in the main result is highlighted with the simple example: set $n = 1, \dimR = 1$ and $F(x) = x^2$. Taking $\eta = \frac{1}{2}$, the SGD dynamics revert to deterministic gradient descent $X_{k+1} = \frac{1}{2} X_k$.  If $X_0 = 1$, so that $\mu_0 = \delta_{1}$ is a Dirac mass, then $\mu_k \rightarrow \delta_0$ converges weakly as $k \rightarrow 0$, but does not converge in the metric used in the main result.

A practically relevant setting where convergence is established without the assumption \ref{A5} occurs when each $f_\findex$ is convex (not necessarily separable) and shares a simultaneous critical point $x^*$. Then $\delta_{x^*}$ is an invariant measure of $\mathcal{P}$, and an application of Hutchinson \cite{Hutchinson81} can directly be used to show an initial measure $\mu_0$ converges geometrically to $\delta_{x^*}$ in the Wasserstein metric.  If $x^*$ is a local minimum for some $f_\findex$ and a saddle or local maximum for others, then general sufficient conditions for convergence of $\mu_k$ to $\delta_{x^*}$ are more subtle. 

%==============================================================================
\subsection{Main Result}\label{sec:MainResult}
%==============================================================================
%==============================================================================
In this section we outline our main result; some technical definitions are deferred to \S\ref{Sec:MathBackground}.

Our starting point is to define formulas for disjoint closed rectangles $T_{\Tindexd}$ that, as we will show, are \emph{absorbing sets}. With positive probability the SGD dynamics $X_k$ will reach one of these sets within a finite number of steps. Each $T_\Tindexd$ will contain exactly one invariant measure for the SGD Markov chain. 

The first step is to define disjoint closed intervals $T_\Tindex$ in dimension $\dimR = 1$ in terms of sets $L$ and $R$ characterizing left and right moving dynamics. For notational brevity we drop the superscripts in $f_\findex^{(1)}$ and $\varphi_\findex^{(1)}$ in this $\dimR = 1$ setting.

When $d=1$, the maps $\varphi_\findex$ have the property that, for all $\eta > 0$,
\begin{align*}
    \varphi_\findex(x) > x \quad \textrm{if} \quad f_\findex'(x) < 0 \, 
    \qquad \textrm{and} \qquad 
    \varphi_\findex(x) < x \quad \textrm{if} \quad f_\findex'(x) > 0 \, .
\end{align*}
Hence, $\varphi_\findex$ maps the point $x$ to the left when $f_\findex'(x) > 0$ and to the right when $f_\findex'(x) < 0$; $\varphi_\findex(x) = x$ is a fixed point when $f_\findex'(x) = 0$.  This motivates defining the following \emph{left} and \emph{right} sets as
\begin{equation}
\label{L_def}
    L := \bigcup_{\findex=1}^n \left\{ x \in \mathbb{R} \; : \; f'_\findex(x) > 0 \right\} \, , 
\end{equation}
and 
\begin{equation}
\label{R_def}
    R := \bigcup_{\findex=1}^n \left\{ x \in \mathbb{R} \; : \; f'_\findex(x) < 0 \right\} \, . 
\end{equation}
Note that if $X_k \in L$ then there is a non-zero probability that the SGD iterate can move to the left, e.g., $X_{k+1} <  X_k$ with positive probability (an analogous result holds for $X_k \in R$).  The sets $L$ and $R$ also characterize when the SGD dynamics move to the left or right with probability one:  A point $x \notin L$ if and only if 
\begin{align}\label{not_inL}
    \varphi_\findex(x) \geq x \, , \qquad \textrm{for all }  1 \leq \findex \leq n \, ,
\end{align}
and $x \notin R$, if and only if 
\begin{align}\label{not_inR}
    \varphi_\findex(x) \leq x \, , \qquad \textrm{for all }  1 \leq \findex \leq n \, .
\end{align}
Several properties of $L$ and $R$ are established in \S~\ref{Sec:PropertiesLRT}.  We now define the sets $T_\Tindex$. 

\begin{definition}[The sets $T_\Tindex$]\label{Def:SetsTj} For a collection of functions $\{f_i\}_{i=1}^n$ in dimension $\dimR = 1$ with $L$ and $R$ given in \eqref{L_def}--\eqref{R_def}, define the sets $T_\Tindex$ to be closed intervals $[l, r]$ ($l < r$) satisfying
    \begin{align*}
        (l, r) \subset L \cap R \, 
    \end{align*}   
    where $l \in \partial L$, $r \in \partial R$.     
    Let $\numT$ be the number of such sets and denumerate them as 
    \begin{align*}
        T_\Tindex = [l_\Tindex, r_\Tindex] \qquad \Tindex = 1, \ldots, \numT \, .
    \end{align*}        
\end{definition}
 
Intuitively, the $T_\Tindex$'s are constructed first by taking the intersection $L$ with $R$ and keeping only the intervals for which $l \in \partial L$ and $r \in \partial R$.  In the subsequent theorem, the definition of $T_\Tindex$ as the closure of $(l, r)$ accounts for cases where $f^\prime_\findex$ may fail to change sign on either side of a critical point. 

\cref{Prop:FiniteTs} in \S\ref{Sec:PropertiesLRT} will establish that the sets $T_\Tindex$ exist $(\numT \geq 1)$, are disjoint, and without loss of generality can be ordered, e.g., $r_\Tindex < l_{\Tindex+1}$ for $1 \leq \Tindex \leq \numT - 1$.

Building on the one-dimensional setting, we extend the definition of the sets $T_\Tindex$ to the multivariate case of separable functions as follows. 
 For each $1 \leq \dimindex \leq \dimR$, define the sets 
\begin{align*}
     T_1^{(\dimindex)} \, , \; T_2^{(\dimindex)} \, , \; \ldots \,, \; T_{\M_\dimindex}^{(\dimindex)} \,, 
\end{align*}
by applying the \cref{Def:SetsTj} to the family of functions $\left\{f_\findex^{(\dimindex)} \right\}_{\findex=1}^n$, where now we denote the number of such sets by $\M_\dimindex$ (instead of $\numT$).

Let $\mathcal{{\M}}:= [\M_1] \times  [\M_2] \hdots \times [\M_d]$ be a subset of integer tuples.  For each \newline $\Tindexd=(\Tindex_1,\Tindex_2, \hdots, \Tindex_d)  \in \mathcal{\M}$ we then define the rectangles $T_{\Tindexd}$ (see \cref{Fig_absorbing_rect}) which are the multivariate generalizations of $T_\Tindex$ as
\begin{equation}\label{Eq:DefT_Rd}
    T_{\Tindexd} :=  T_{\Tindex_1}^{(1)} \times T_{\Tindex_2}^{(2)} \hdots \times T_{\Tindex_d}^{(d)} \subset \mathbb{R}^{\dimR} \, .
\end{equation}

    The union of all such sets is
    \begin{equation}\label{Eq:T_Union}
        T :=  \bigcup_{ \Tindexd \in \mathcal{\M}}   T_{\Tindexd} \subset \mathbb{R}^{\dimR} \,. 
    \end{equation}    
\begin{figure}[htb!]
      \centering
  \includegraphics[scale=.7]{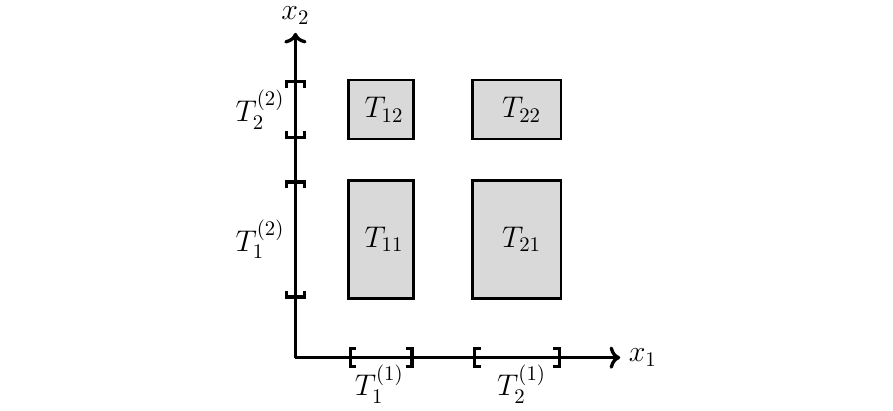}
    \caption{Sketch of the rectangles $T_{\boldsymbol{m}}$.}\label{Fig_absorbing_rect}
 \end{figure}
 
In the subsequent proofs, it will also be useful to define the following subsets in $\mathbb{R}$ for each variable $x_\dimindex$
\begin{equation*}
     T^{(\dimindex)} :=  \bigcup_{\Tindex=1}^{\M_\dimindex} T_\Tindex^{(\dimindex)}  \subset \mathbb{R} \, ,
\end{equation*}
With this notation, the set $T \in \mathbb{R}^d$ admits an alternative form
\begin{equation}
    \label{T_Rd}
    T= T^{(1)} \times T^{(2)} \hdots \times T^{(\dimR)}.
\end{equation}

We can now state the main result. 
\begin{theorem}[Main Result]\label{Main_thm_basin} Given the Markov chain \eqref{Eq:MarkovOperatorDynamics} and \eqref{Eq:SGD_Markov_Operators} corresponding to the SGD dynamics \eqref{F}--\eqref{Eq:SGDIterates}, assume \ref{A1}--\ref{A5} hold. Let $I$ and $T_{\Tindexd}$ be defined as in \eqref{Def:Statespace} and \eqref{Eq:DefT_Rd}, and let $\eta$ be any value $0 < \eta < 1/\LipK$.  Then,
\begin{enumerate}[label=(\alph*) ]
    \item \label{Main_a} The statespace $I$ is positive invariant and decomposes into the disjoint union 
        \begin{align*}
            I = B \cup \left( \bigcup_{ \Tindexd \in \mathcal{\M}} T_{\Tindexd} \right)  \, \qquad (B := I \setminus T) \, ,
        \end{align*}
        where 
        \begin{enumerate}[label=(\roman*)]
            \item \label{Main_a_i} $T_\Tindexd$ ($\Tindexd \in \mathcal{\M}$) is positive invariant/absorbing, and contains at least one local minimizer of $F$.  The number of $T_{\Tindexd}$ is at most the number of local minima of $F$.
            \item \label{Main_a_ii} $T$ is non-empty and there exists $\ell_0 = \ell_0(\eta)$ such that for every $x \in I$ there is a path $\vec{p}$ of length $\ell_0$ satisfying $\varphi_{\vec{p}\,}(x) \in T$.
            %\revAB{ $\varphi_{\vec{p}}(x) \in \inter T$, where $\inter T$ denotes the interior of $T$}.
            %
            \item \label{Main_a_iii} $B$ is uniformly transient with
            \begin{align}\label{Eq:Mass_On_B}
                ({\mathcal{P}}^{\ell_0} \mu) (B) \leq \left( 1 - \frac{1}{n^{\ell_0}} \right) \, \mu(B) \, ,      
            \end{align}         
            for any probability measure $\mu \in \mathscr{P}(I)$. 
      \end{enumerate}
    \item \label{Main_b} For each $\Tindexd \in \mathcal{\M}$, there exists a unique invariant measure $\mu_\Tindexd^\star$ with support contained in $T_\Tindexd$.
    %$\mathrm{supp}\mu^{\star} \subset T_\Tindexd$
    Moreover, there exist an $\ell_{\Tindexd} \in \mathbb{N}$, {$\ell_{\Tindexd} \geq \dimR$} and $\alpha_{\Tindexd} \in \{+1, -1\}^{\dimR}$ such that for any $\mu$ with support contained in $T_\Tindexd$, we have 
    \begin{align} 
        \label{Con_Tr_Rd}
            \begin{split}
                d_{\dtuple_{\Tindexd}} \left( {\mathcal{P}}^{k} \mu, \mu_{\Tindexd}^\star \right) 
                 \leq \left( 1- \frac{1}{n^{\ell_\Tindexd}}\right)^{\lfloor k/\ell_\Tindexd \rfloor} \qquad\qquad k > 0 \, . 
        \end{split}
    \end{align}       
    In the case of $\dimR = 1$, the convergence \eqref{Con_Tr_Rd} is a (stronger) contraction
    \begin{align}\label{Eq:CvgInTj}
            \begin{split}
                d_F \left( {\mathcal{P}}^{\ell_{\Tindex} } \mu, \mu_\Tindex^\star \right) 
                \leq \left( 1- \frac{1}{n^{\ell_\Tindex}}\right) \, d_F(\mu, \mu_\Tindex^\star)     \, .
        \end{split}
    \end{align}
    Here the metrics $ d_{\dtuple_{\Tindexd}} $ and $d_F$ are defined in \cref{Sec:MathBackground} and $\lfloor \cdot \rfloor$ is the greatest integer (floor) function.
    The measures $\mu_{\Tindexd}^\star$ are the only invariant measures supported in $I$, and the number $\numT$ is bounded by the number of local minima of $F$. 

%    There is a unique set of unique continuous functions $\{ g_{\Tindexd}\}_{\Tindexd \in \mathcal{M}}$ that: 
    %\begin{enumerate}
    %    \item \label{Main_c_i} Form a partition of unity on $I$ (i.e., for each $x \in I$, $0 \leq g_{\Tindexd}(x) \leq 1$ and $\sum_{\Tindexd \in \mathcal{M}} g_{\Tindex}(x) = 1$);
    %    \item \label{Main_c_ii} Satisfy $g_{\Tindexd}(x) = 1$ for all $x \in T_{\Tindexd}$;
    %    \item \label{Main_c_iii} Are eigenvectors of $\mathcal{P}^{\dag} g_{\Tindexd} = g_{\Tindexd}$. 
    %\end{enumerate}
    % For each $\Tindexd \in \mathcal{\M}$, there is a unique continuous function $g_{\Tindexd}$ 

     \item \label{Main_result_d} For any probably measure $\mu_0 \in \mathscr{P}(I)$, $\mu_k := {\mathcal{P}}^{k} \mu_0$ converges geometrically to the invariant measure {
    \begin{align*}
        \mu^{\star} = \sum_{\Tindexd \in \mathcal{\M}} c_{\Tindexd}(\mu_0) \, \mu_{\Tindexd}^{\star} \qquad
        \textrm{where} \qquad c_{\Tindexd}(\mu_0) = \int_{I} g_{\Tindexd}(x) \, \textrm{d}\mu_0(x) \,. 
        %\left(\revA{\mu_0(T_\Tindexd)} \leq c_{\Tindexd} \leq 1, \quad \sum_{\Tindexd \in \mathcal{\M}} c_{\Tindexd} = 1\right) \, ,
    \end{align*}    
    Here $\{ g_{\Tindexd}\}_{\Tindexd \in \mathcal{M}}$ exist and are the unique set of functions satisfying the following:    
    \begin{enumerate}[label=(\roman*)]
        \item\label{Main_d_i} Each $g_{\Tindexd} : I \rightarrow \mathbb{R}$ is a left eigenvector of $\mathcal{P}$, i.e., $\mathcal{P}^{\star} g_{\Tindexd} = g_{\Tindexd}$.
        \item\label{Main_d_ii} They form a partition of unity for the set $I$, i.e., the $g_{\Tindexd}$ are continuous, non-negative functions with $\sum_{\Tindexd \in \mathcal{M}} g_{\Tindexd}(x) = 1$ for each $x \in I$.
        \item\label{Main_d_iii} The functions satisfy $g_{\Tindexd}(x) = 1$ for all $x \in T_{\Tindexd}$, and zero on $T_{\Tindexd'}$ for $\Tindexd' \neq \Tindexd$. In particular, $\int_{I} g_{\Tindexd}(x) \, \textrm{d}\mu^{\star}_{\Tindexd'}(x) = \delta_{\Tindexd \Tindexd'}$. 
    \end{enumerate}}
    Convergence of $\mu_k \rightarrow \mu^{\star}$ is in the sense that    
    %\item \label{Main_result_d} For any probably measure $\mu_0 \in \mathscr{P}(I)$, $\mu_k := {\mathcal{P}}^{k} \mu_0$ converges geometrically to an invariant measure of the form
    %\begin{align*}
    %    \mu^{\star} = \sum_{\Tindexd \in \mathcal{\M}} c_{\Tindexd} \, \mu_{\Tindexd}^{\star} \qquad
    %    \left(\revA{\mu_0(T_\Tindexd)} \leq c_{\Tindexd} \leq 1, \quad \sum_{\Tindexd \in \mathcal{\M}} c_{\Tindexd} = 1\right) \, ,
    %\end{align*}    
    %in the sense that
    \begin{equation}
    \label{Eq:MainCvgMetric}
        \Tilde{d}(\mu_k, \mu^\star) \leq 3 \left( 1- \frac{1}{n^{\ell}}\right)^{ \lfloor k/\ell \rfloor}  \qquad\qquad k > 0 \, .
    \end{equation}
    Here $\ell := 2 \, ( \ell_0 \vee \max_{\Tindexd \in \mathcal{\M}} \ell_\Tindexd) \, {\geq \dimR}$, and the metric $\tilde{d}$ is defined on probability measures $\mu, \nu \in \mathscr{P}(I)$ as 
            \begin{align} 
 \label{d_big_def}
           \tilde{d}(\mu,\nu) :=   
                     d_{\rm TV}\big( \mu|_B, \nu|_B\big) +  \sum_{\Tindexd \in \mathcal{\M} }  
                    d_{\alpha_{\Tindexd}}\big(   \mu|_{T_{\Tindexd}},    
                \,  \nu|_{T_{\Tindexd}}\big)  \, 
                  \, ,
        \end{align}    
       where $\nu\big|_{A}$ is the restriction of the finite measure $\nu$ to a Borel set $A$. 
    Lastly, in the case where $d=1$, \eqref{Eq:MainCvgMetric} becomes 
    \begin{equation*}
        d_F(\mu_{k }, \mu^\star) \leq 3 \left( 1- \frac{1}{n^{\ell}}\right)^{\lfloor k/\ell \rfloor}  \qquad\qquad k > 0 \, .
    \end{equation*}
    \end{enumerate}    
\end{theorem}
{\begin{corollary}[Unique invariant measure]\label{Coro_unique}
    Given the same assumptions as \cref{Main_thm_basin}, assume further that for each dimension $j \in [\dimR]$ there exists at least one $f_{i}^{(j)}$ that has a single critical point.  Then, there is exactly one absorbing set $T_{\mathbf{m}}$, and the Markov operator $\mathcal{P}$ has a unique invariant measure $\mu^\star \in \mathscr{P}(I)$.
\end{corollary}}

A few remarks are in order:
\begin{itemize}     
    \item The decomposition of $I$ in \cref{Main_thm_basin}(a) is in the spirit of a ``Doeblin decomposition'' (cf. \cite{MeynTweedie1993}); 
    \item The theorem bounds the support of $\mu_{\Tindexd}^{\star}$ to lie in $T_\Tindexd$ and does not imply the support is equal to $T_\Tindexd$. For instance, when $F$ is quadratic the invariant measure may be supported on a set of Lebesgue measure zero (cf. infinite Bernoulli convolutions \S\ref{Subsec:SGD_Background_ConstEta});
    \item If we let $W$ be the set of invariant measures of $\mathcal{P}$ with support in $I$, then the extreme points of $W$ are exactly $\mu_\Tindexd^\star$ ($\Tindexd \in \mathcal{\M}$);
    %\item The theorem shows that if SGD is initialized to $X_0 \in T_\Tindexd$ then $X_k$ samples $\mu_\Tindexd^\star$;
    \item Since the $T_{\Tindexd}$ are disjoint closed rectangles, \cref{Main_thm_basin} implies the supports of every pair of distinct invariant measures may be separated by a hyperplane, i.e., the supports of any pair of invariant measures are not ``interlaced;''    
    \item While the number of sets $T_{\Tindexd}$ (and invariant measures) is bounded in terms of $F$, the definition of $T_\Tindexd$ depends on how $F$ is split into the functions $f_\findex$;
    \item The number of invariant measures does not depend on $\eta$ provided $\eta < 1/\LipK$, and the associated Markov chain for SGD exhibits no bifurcations as a function of $\eta$. The rate of convergence to equilibrium does however depend on $\eta$;
    \item The SGD iterates $X_k$ may traverse between two local minima of $F$ only if both minima are contained in the same $T_{\Tindexd}$ (somewhat akin to a ``mountain pass'' theorem);
    \item Somewhat surprisingly, the converse to \cref{Main_thm_basin}(a)(i) need not hold: the local (and even global) minima of $F$ do not need to be in $T$ (see \S\ref{Sec:Examples1d}). In other words, there are instances of SGD where even if the iterates $X_k$ are initialized to lie in a neighborhood of the global minimum of $F$, they will eventually leave (with probability $1$);
    \item The contraction estimate \eqref{Eq:CvgInTj} is sometimes referred to as a spectral gap estimate (in the metric $d_F$);
    \item When $\dimR > 1$, the invariant measures $\mu_{\Tindexd}^\star$ are not necessarily product measures. In addition, the main result in dimension $\dimR > 1$ does not follow as a corollary of the one dimensional case;
    %\item An explicit bound on $\ell$ in \cref{Main_thm_basin} determining the convergence rate can be obtained in terms of $f_i^{(j)}$, e.g., see \S\ref{Sec:Examples1d}.   \revB{In addition, $\ell$ is qualitative and may depend on the dimension of the problem $\dimR $ as well as other properties of the functions $F$ and $f_i$. More specifically, an induction argument is used in the proof wherein bounds on $\ell$ grow as the dimension $\dimR $ increases, leading to geometrically worse convergence rates for larger $\dimR $.}   
    %
    \item {The existence proof of $\ell$ in \cref{Main_thm_basin} yields a lower bound of $\ell \geq d$ (where $d$ is the dimension);} it also yields an upper bound in terms of $f_i^{(j)}$, e.g., see \S\ref{Sec:Examples1d}. {Note that the geometric convergence rate in the right hand side of \eqref{Con_Tr_Rd} is no better than $(1 - \frac{1}{n^{d}})^{\lfloor k/d\rfloor }$, since larger $\ell$ values yield slower bounds.}
    %\revB{The proof of \cref{Main_thm_basin} implies a lower bound on $\ell \geq d$ (where $d$ is dimension of $F$) and an upper bound in terms of $f_i^{(j)}$, e.g., see \S\ref{Sec:Examples1d}. Since larger $\ell$ values yield slower convergence rate bounds, the geometric convergence in the right hand side of \eqref{Con_Tr_Rd} is no better than $(1 - \frac{1}{n^{d}})^{\lfloor k/d\rfloor }$.}
    %\begin{align*}
    %    \sim (1 - \frac{1}{n^{d}})^{1/d}
    %\end{align*}  
    %A bound on $\ell$ is implied by the proof of \cref{Main_thm_basin}. An explicit bound on $\ell$ in \cref{Main_thm_basin} determining the convergence rate can be obtained in terms of $f_i^{(j)}$, e.g., see \S\ref{Sec:Examples1d}.   \revB{In addition, $\ell$ is qualitative and may depend on the dimension of the problem $\dimR $ as well as other properties of the functions $F$ and $f_i$. More specifically, an induction argument is used in the proof wherein bounds on $\ell$ grow as the dimension $\dimR $ increases, leading to geometrically worse convergence rates for larger $\dimR $.}
\end{itemize}

\section{Examples in 1D}\label{Sec:Examples1d}
%==============================================================================
This section provides examples highlighting \cref{Main_thm_basin} in $\dimR = 1$. In each example, we analyze an objective function $F$ with the specific splitting 
\begin{equation}
\label{F_splitting_ex}
    F(x) = \frac{1}{2} \big(f_1(x) + f_2(x) \big) \, ,
\end{equation}
where 
\begin{equation}
\label{f_1f_2}
    f_1(x) := F(x) + \parval x \, ,  \qquad \textrm{and} \qquad f_2(x) :=F(x) - \parval x \, .
\end{equation}
Here $\parval > 0$ is a free parameter that modifies the functions $f_1, f_2$ in the decomposition. In this case, the maps $\varphi_1$ and $\varphi_2$ become 
\begin{equation*}
    \varphi_1(x) =x- \eta F^\prime (x) - \parval \eta \, ,  \qquad \textrm{and}  \qquad \varphi_2(x) :=x- \eta F^\prime (x) + \parval \eta  \, .
\end{equation*}
Notice that for this splitting, the SGD update becomes gradient descent plus an additional random walk with step-size $\lambda \eta$, and is also the one dimensional analog of the algorithm proposed in \cite{Jin2007}.

%==============================================================================
\subsection{Diffusion Approximation Background}\label{sec:DiffApprox}
%==============================================================================
    We collect here several basic facts regarding the diffusion approximation in dimension $\dimR = 1$ as we reference it in subsequent examples.
        
    The diffusion approximation is a variable coefficient advection-diffusion equation of the form
    \begin{align}\label{Eq:TimeEvolutionPDE1d}
        \frac{\partial \rho}{\partial t} 
        &= \frac{\partial}{\partial x}\big( u(x) \rho \big) + \frac{\eta}{2} \frac{\partial^2}{\partial x^2}\Big( D(x) \rho \Big) \, 
        \qquad \textrm{in} \quad (x,t) \in \mathbb{R} \times (0, T] \, ,
    \end{align}
    with initial data $\rho(x,0) = \mu_0$, where $\mu_0$ is the SGD initialization distribution. 
    
    The velocity $u(x)$ and diffusion coefficient $D(x)$ in \eqref{Eq:TimeEvolutionPDE1d} are given in terms of the SGD functions $f_\findex(x)$ and $F(x)$ as
    \begin{align*}          
        u(x) := \frac{d}{dx} \Phi(x) \qquad \textrm{where} \qquad \Phi(x) := F(x) + \frac{\eta}{4}(F'(x))^2 \, , 
    \end{align*}
    and
    \begin{align*}       
        D(x) &:= \frac{1}{n} \sum_{\findex=1}^n \Big( f_\findex^{\prime}(x) - F^{\prime}(x) \Big)^2 
        = \frac{1}{n} \sum_{\findex=1}^n \big(f_\findex^\prime(x) \big)^2 - \big(F^{\prime}(x)\big)^2 \geq 0 \, .
    \end{align*}
    Intuitively, the advective term in \eqref{Eq:TimeEvolutionPDE1d} evolves the probability $\rho$ towards minizers of $F$, while the diffusion term arises from the stochastic terms $f_\findex$ in SGD. For instance, formally setting $\eta = 0$ in \eqref{Eq:TimeEvolutionPDE1d} results in an advection equation for $\rho$, with \emph{characteristics} defined by the gradient flow $\dot{x} = -F^\prime(x)$.
    
    The equation \eqref{Eq:TimeEvolutionPDE1d} arises as a formal asymptotic approximation to the (exact) discrete-in-time Markov evolution $\mu_{j+1} = \mathcal{P} \mu_j$ in the small parameter $\eta \ll 1$ by matching terms up order $\mathcal{O}(\eta)$. When $\eta \ll 1$, $\rho(x,t)$ $(t = \eta j)$ approximates the SGD probability evolution $\mu_j$ for finite times (cf. \cite{LiTaiE2017, FengLiLiu2018, FengGaoLiLiuLu2020, HuLiLiLiu2017}).
    
    The stationary solutions of \eqref{Eq:TimeEvolutionPDE1d} satisfy
    \begin{align}\label{Eq:SteadyStatePDE1d}
        \frac{d}{dx}\big( u(x) \rho \big) + \frac{\eta}{2} \frac{d^2}{dx^2} \Big( D(x) \rho \Big) = 0 \, .
    \end{align}
    Note that \eqref{Eq:SteadyStatePDE1d} is a singular ordinary differential equation whenever the diffusion coefficient $D(x)$ vanishes at a point $x^*$, namely
    \begin{align}\label{Eq:VanishingD}
        D(x^*) = 0 \qquad \Longleftrightarrow \qquad f_1^{\prime}(x^*) = f_2^{\prime}(x^*) = \ldots = f_n^{\prime}(x^*) = F^{\prime}(x^*) \, . 
    \end{align}
%    Under Assumption \ref{A5}, $D(x) \neq 0$ since the expression in the right of \eqref{Eq:VanishingD} holds nowhere. 
        
   If \eqref{Eq:VanishingD} holds nowhere, then  equation \eqref{Eq:SteadyStatePDE1d} admits the following unique solution in the space of probabilities densities 
    \begin{align}
    \label{density_DA}
        \rho^*(x) &= Z^{-1} \, \textrm{exp}\Big( - \frac{2}{\eta} V(x) \Big) \,  ,
    \end{align}
    where
    \begin{align*}        
        V(x) &:= \int^x D^{-1}(x) \frac{d}{dx}\Big( \Phi(x) + \frac{2}{\eta} D(x) \Big) \, dx \, ,  
    \end{align*}
    provided 
    \begin{align*}
        Z := \int_{-\infty}^{\infty} \textrm{exp}\Big(- \frac{2}{\eta} V(x) \Big) \, dx  < \infty \, .
    \end{align*}             
    For the splitting given in \eqref{F_splitting_ex} and \eqref{f_1f_2}, the stationary density of the diffusion approximation given in \eqref{density_DA} is simplifies to
\begin{equation}
\label{Rho_examples}
    \rho^\star(x) \propto \textrm{exp} \left( -\frac{2}{\eta \lambda^2} \Phi(x) \right) =  \textrm{exp} \left( -\frac{2}{\eta \lambda^2} \left( F(x) + \frac{\eta}{4}(F'(x))^2 \right) \right) \, .
\end{equation}
In the following examples, we will compare the approximation \eqref{Rho_examples} to the true invariant measures of the SGD Markov operator $\mathcal{P}$. We will see that for larger values of $\lambda$ the Markov operator $\mathcal{P}$ has a unique invariant measure $\mu^\star$ which is approximated by the density $\rho^\star(x)$. However, as $\lambda$ decreases, the Markov operator $\mathcal{P}$ may have multiple invariant measures that differ from the diffusion approximation. 
%==============================================================================
\subsection{SGD on a One Dimensional Double Well}
\label{Subsec:DoubleWell}
%==============================================================================
This first example demonstrates \cref{Main_thm_basin} with SGD applied to the double-well objective function
\begin{align}
\label{DW}
        F(x) := \frac{1}{4}  (1-x^2)^2 \, ,
\end{align}
with the splitting given in \eqref{F_splitting_ex} and \eqref{f_1f_2}.
The free parameter $\parval > 0$ modifies the functions $f_1, f_2$ in the decomposition (see  \cref{DW_plot}) and will lead to a bifurcation in the invariant measures of the associated Markov operator.

\begin{figure}
\centering
\begin{minipage}[b]{.9\textwidth}
\includegraphics[trim={1.2cm 0 0 0},clip,width=\textwidth]{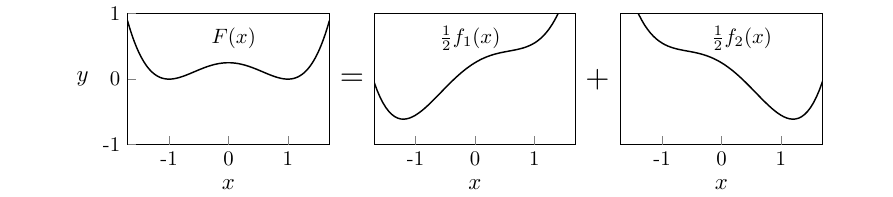}
\end{minipage}
%\hfill
\begin{minipage}[b]{.9\textwidth}
\includegraphics[trim={1.2cm 0 0 0},clip,width=\textwidth]{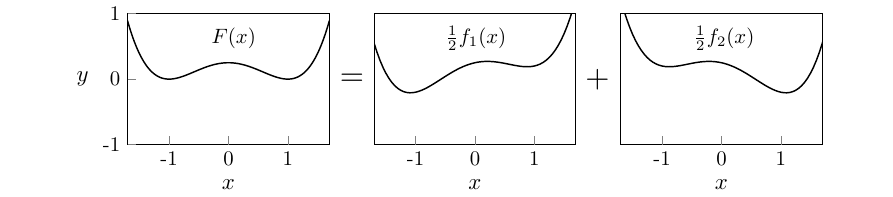}
\end{minipage}
 
\caption{Visualization of the SGD model problem given by \eqref{F_splitting_ex}--\eqref{f_1f_2} and \eqref{DW} for values (Top) $\lambda=.55> \lambda_c$, and (Bottom) $\lambda=.2 < \lambda_c$.  When $\lambda > \lambda_c$, the SGD iterates can cross over the barrier of $F$ and there is a unique invariant measure. When $\lambda < \lambda_c$ the SGD iterates cannot cross over the barrier of $F$ and there are two invariant measures. \label{DW_plot}  }
\end{figure} 

We first establish several observations to invoke \cref{Main_thm_basin}.  The critical points of $f_2$, and by symmetry $f_1$ (letting $x \mapsto -x$), are solutions to the cubic equation
\begin{align*}    
    x^3 - x - \lambda = 0 \, .
\end{align*}
Denote $x_0 = x_0(\lambda) > 0$ as the largest solution; thus $I = [-x_0, x_0]$.  

The number of critical points in each of $f_1, f_2$ change as a function of $\lambda$ from $1$ to $3$ (see \cref{Fig:DW_left_right_sets}) with the critical value being
\begin{equation*}
\label{L_c}
    \parval_c := \frac{2}{3 \sqrt{3}} \, . 
\end{equation*}
Note also that for all $\lambda > 0$ the critical points of $f_1$ and $f_2$ are distinct.

The Lipschitz constant of $f_\findex^{\prime}$ on $I$ is:
\begin{align}
    \textrm{Lip} \; f_\findex^{\prime} = 3 x_0^2 - 1 \, \qquad (\findex = 1, 2)\, ,
\end{align}
which sets the upper bound on $\eta$ (see \cref{Supp_DW_fig}) given in \cref{Main_thm_basin} as
\begin{align}
\label{Etabbd}
    \eta_0(\lambda) = \big( 3  x_0^2 - 1 \big)^{-1}  \,. 
\end{align}
Thus, \cref{Main_thm_basin} applies for all $\lambda > 0$ and $0 < \eta < \eta_0(\lambda)$. 
\begin{figure}[htb!]
      \centering
  \includegraphics[scale=.9]{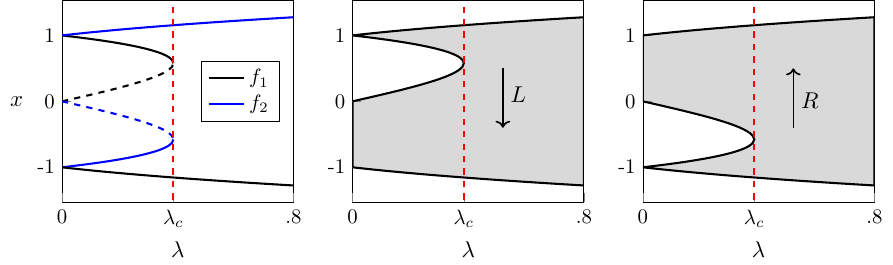}
\caption{Model problem \eqref{DW}.  Left: the critical points of the functions $f_1$ (black) and $f_2$ (blue) are plotted as a function of $\lambda$.  Solid curves represent local minima and dashed curves represent local maxima. Middle: for each $\lambda$ the vertical cross section of the filled in region represents the left moving set $L$. Right: for each $\lambda$ the vertical cross section of the filled in region represents the right moving set $R$.      }\label{Fig:DW_left_right_sets}
 \end{figure}
 
\begin{figure}[htb!]
      \centering
  \includegraphics[scale=.9]{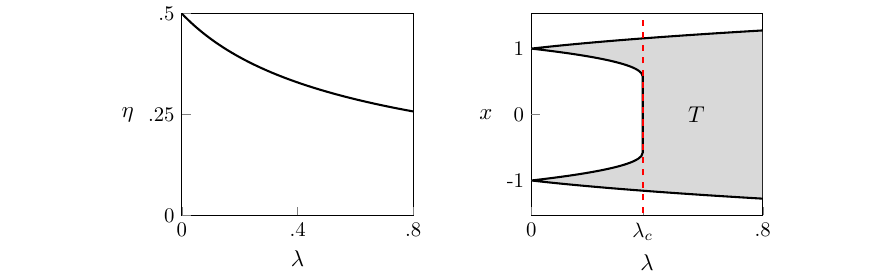}
  \caption{Model problem \eqref{DW}. Left: the bound on $\eta$ for the validity of \cref{Main_thm_basin} given in \eqref{Etabbd} is plotted as a function of $\lambda$. Right: for each $\lambda$ the vertical cross section of the filled in region represents the absorbing set $T$. For $\lambda> \lambda_c$ we have one absorbing interval and for $\lambda \leq \lambda_c$ we have two absorbing intervals.  }
  \label{Supp_DW_fig}
 \end{figure} 

\medskip

\noindent\fbox{Case 1: $\lambda > \lambda_c$.} The functions $f_1, f_2$ have a unique critical point given by 
    \begin{align*}
        f_1^{\prime}(-x_0) = 0 \qquad\qquad f_2^{\prime}(x_0) = 0 \, ,
    \end{align*}
    and a single $T_1 = I = [-x_0, x_0]$ (see \cref{Supp_DW_fig}) defined by the sets $L$ and $R$ presented in \cref{Fig:DW_left_right_sets}.  By \cref{Main_thm_basin} the operator $\mathcal{P}$ has a unique invariant measure $\mu^\star$ with support contained in $T_1$.  \cref{DW_Mu1a} visualizes the crude agreement between the diffusion approximation $\rho^\star(x)$ defined by \eqref{Rho_examples}--\eqref{DW} and a numerical approximation to $\mu^{\star}$. 

    \par 
    In addition, from \cref{Main_thm_basin} any initial $\mu_0$ supported on $I=T_1$ converges via
    \begin{equation}\label{Eq:SpectralGap}
        d_F(\mu_{k }, \mu^\star) \leq  \left(1-\frac{1}{2^{\ell_1}} \right)^{ \lfloor k/\ell_1 \rfloor } d_F(\mu_0, \mu^\star) \, .
    \end{equation}
    
    A bound on $\ell_1$, and hence the convergence rate, can be obtained from \cref{Thm:DB_Path} with the two paths satisfying \eqref{Path_condition} taken to be $\vec{i}$ (resp. $\vec{j}$) as the $\ell_1$ consecutive compositions of $\varphi_2$ (resp. $\varphi_1$).  For all $x \in [-x_0, 0]$ the map $\varphi_2$ moves the point $x$ to the right by the amount  
    \begin{align*}
       \varphi_2(x) - x &= \eta \big(x-x^3+ \parval \big) \\
       &\geq \eta \big( \parval - \parval_c \big) \, , 
   \end{align*}
    which follows by substituting the minimum $x=-\frac{1}{\sqrt{3}}$.
    
    Hence, by symmetry of the maps, condition \eqref{Path_condition} is satisfied as 
    \begin{align*}
        \varphi_{\vec{i}\,}(-x_0) \geq 0 \geq \varphi_{\vec{j}\,}(x_0) \, ,
    \end{align*}
    by taking  
    \begin{equation}\label{Eq:BoundM1}
        \ell_1 = 1 + \left \lfloor  \frac{x_0}{\eta ( \parval - \parval_c )} \right \rfloor 
        \leq 1 + \left \lfloor  \frac{1 + \sqrt[3]{\parval}}{\eta ( \parval - \parval_c ) } \right \rfloor \, .
   \end{equation}     
    Notice the upper bound on $\ell_1$ given in \eqref{Eq:BoundM1} approaches $\infty$ as $\lambda \rightarrow \lambda_c$.  While \eqref{Eq:SpectralGap} together with \eqref{Eq:BoundM1} is only a lower bound on the spectral gap, the critical slow down in convergence rate as $\lambda \rightarrow \lambda_c$ is also observed numerically.

\begin{figure}[htb!]
      \centering
  \includegraphics[scale=.9]{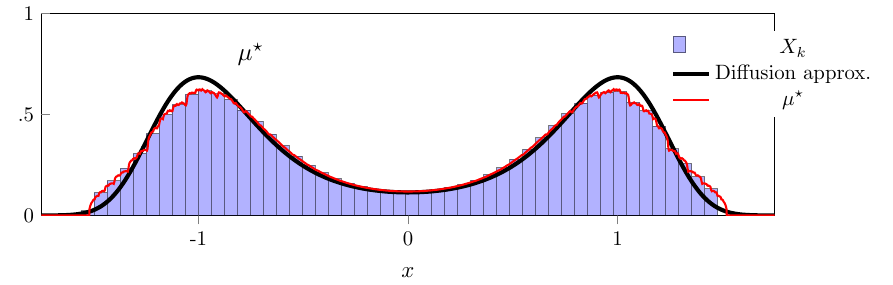}
  \caption{Model problem \eqref{DW}. A comparison of the exact unique invariant measure $\mu^\star$ (red) and diffusion approximation \eqref{Rho_examples}--\eqref{DW} (black) for parameter values $\parval=2 > \parval_c$, $\eta=.0698$ (satisfying \eqref{Etabbd}). Here Ulam's method \cite{Ulam_method} is used to numerically compute the exact invariant measure (red). For comparison a time histogram of the iterates $X_k$ are plotted (blue) showing good agreement with $\mu^\star$. Note that the lack of smoothness in $\mu^\star$ (red) is a property of the invariant measure and not a result of under resolved computations.}
   \label{DW_Mu1a}
 \end{figure}

\medskip
\noindent\fbox{Case 2: $\lambda \leq \lambda_c$.} The functions $f_1, f_2$ each have three critical points (see \cref{Fig:DW_left_right_sets}) given by    
    \begin{align*}
        f_2^{\prime}(x_j) = 0 \qquad j = 0,  1,  2 \, \qquad  \textrm{where} \qquad x_2  \leq x_1 < 0 < x_0 \, ,        
    \end{align*}
    and by symmetry $f_1^{\prime}(-x_j) = 0$ for $j = 0, 1, 2$. When $\lambda = \lambda_c$ the two critical points $x_1 = x_2$, while $x_2 < x_1$ when $\lambda < \lambda_c$. In addition,  \cref{Fig:DW_left_right_sets} shows the construction of the sets $L$ and $R$, which leads to the absorbing sets displayed in \cref{Supp_DW_fig}. In this case, there are two absorbing intervals given by
    \begin{align*}
        T_1 = [-x_0, x_2] \qquad \qquad T_2 = [-x_2, x_0] \, . 
    \end{align*}
    
   By \cref{Main_thm_basin} the operator $\mathcal{P}$ has two invariant measures $\mu_1^\star$ and $\mu_2^\star$ (see \cref{DW_mu1_mu2}) supported on $T_1$ and $T_2$ respectively.  Furthermore, {any initial measure $\mu_0$ supported on $I$ converges as
    \begin{equation*}
        d _F\Big( \mu_{k } , \,  c_1 (\mu_0) \mu_1^\star+    c_2(\mu_0) \mu_2^\star \Big)  \leq 3 \left(1 - \dfrac{1}{2^{\ell}}\right)^{\lfloor k/\ell \rfloor} \, ,
    \end{equation*}
where 
\begin{equation}
\label{g_12_example}
    c_1 (\mu_0) : = \int_I g_1(x) \mu_0 (dx) \qquad \textrm{and} \qquad  c_2 (\mu_0) : = \int_I g_2(x) \mu_0 (dx) \, .
\end{equation}
For a specific choice of parameters $\lambda$ and $\eta$, the functions $g_1$ and $g_2$ are plotted in \cref{DW_mu1_mu2}. }Lastly, a bound on $\ell$ can be obtained via elementary means.
    \par
    This example demonstrates the discrepancy between the exact invariant measures (e.g., for which there are two, $\mu_1^\star$ and $\mu_2^\star$) and the single invariant measure $\rho^\star(x)$ predicted by the diffusion approximation. In particular, the exact SGD dynamics cannot escape the local minima of $F$, while the diffusion approximation implies that iterates of SGD will eventually escape and travel between local minima. {A similar observation on the inability of the iterates $X_k$ to cross a saddle was also made for closely related dynamics in \cite[Remark 12]{KongTao2020}}.

    Lastly, the example also demonstrates that the number of invariant measures is at least one and most two, since $F(x)$ has two local minima. While the bounds on the spectral gap and the invariant measures depend on $\eta$, the sets $T_\Tindex$ ($\Tindex = 1,2$) and the number of invariant measures are independent of $\eta$ (provided $\eta < \eta_0$).  As a result, there are no bifurcations in the dynamics $\mu_{k+1} = \mathcal{P} \mu_k$ (i.e., creation or loss of new fixed points) in terms of $\eta$. 

%==============================================================================
% FIGURES
%==============================================================================
\begin{figure}[htb!]
      \centering
      \begin{minipage}{1\textwidth}
      \centering
           \includegraphics[scale=.90]{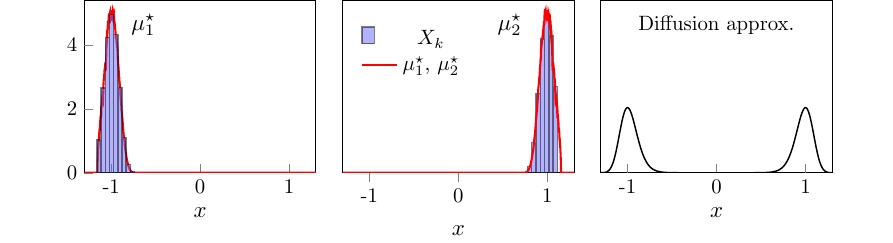}
      \end{minipage}

  \begin{minipage}[b]{1\textwidth}
  \centering
\includegraphics[scale=.9,trim={0cm 0cm .7cm 0},clip]{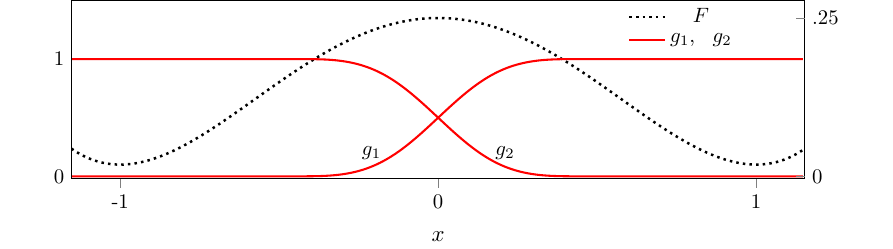}
\end{minipage}

  \caption{Model problem \eqref{DW}. For $\parval=.38 < \parval_c$ and $\eta=.33$ (satisfying \eqref{Etabbd}) there are two invariant measures $\mu_1^\star$ ({top} left), $\mu_2^\star$ ({top} middle), while the diffusion approximation incorrectly predicts a unique stationary distribution $\rho^\star$ ({top} right). Note that the lack of smoothness in $\mu_j^\star$ ($j=1,2$) again is a property of the invariant measure and not due to an unresolved computation. {The functions $g_1$, $g_2$ (red) given in \eqref{g_12_example} along with the objective function $F$ are also plotted on the bottom panel.}  
  }    
 \label{DW_mu1_mu2}
 \end{figure}

%================================================================================    
\subsection{An Example where SGD Does Not Sample the Global Minimum}\label{Subsec:ExGlobalMinFail}
%================================================================================
Here we provide an example where the global minimum of the objective function $F$ is contained inside a uniformly transient region and hence not contained in the support of an invariant measure. From \cref{Main_thm_basin} this implies that as $k \rightarrow \infty$ the iterates $X_k$ do not sample the global minimum (even for small $\eta$ values and arbitrary initializations $X_0$).
\par 
Let the objective function $F(x)$ be the eighth order polynomial,
\begin{equation} 
\label{F_gl_min}
    F(x)= c_4 x^4 + c_6x^6 + c_8x^8 \, ,
\end{equation}
where $c_4 \approx 2.84$, $c_6 \approx -2.94 $, and $c_8 \approx 0.78$ are chosen so that $F(x)$ has two local minima at $x= \pm 1.35$ and two local maxima at $x = \pm 1$. The global minimum lies at $x = 0$, which locally is a quartic and is also the ``flatest'' minima, e.g., see \cref{GL_Min_plot}. 
\par 
We again split $F$ as in \eqref{F_splitting_ex} and \eqref{f_1f_2} depending on the paramter $\lambda$.  From \cref{LR_GL_min} there are two critical values of $\lambda$:
\begin{equation*}
    \lambda_1 \approx 1.47 ,  \qquad \textrm{and}  \qquad \lambda_2 \approx 1.85 \, ,
\end{equation*}
for which the critical points of $f_1$ and $f_2$ have a bifurcation. 

\begin{figure}[htb!]
\centering
\begin{minipage}[b]{.9\textwidth}
\includegraphics[trim={.5cm 0 0 0},clip,width=\textwidth]{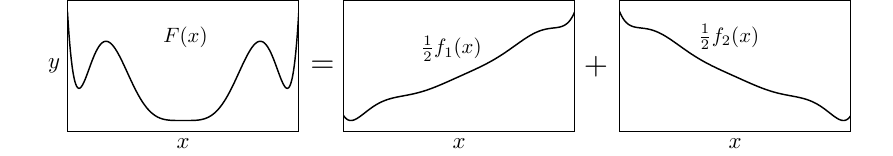}
\end{minipage}
%\hfill
\begin{minipage}[b]{.9\textwidth}
\includegraphics[trim={.5cm 0 0 0},clip,width=\textwidth]{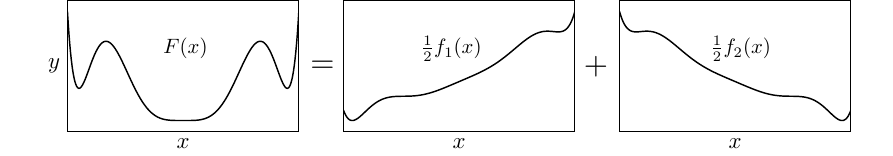}
\end{minipage}
 \begin{minipage}[b]{.9\textwidth}
\includegraphics[trim={.5cm 0 0 0},clip,width=\textwidth]{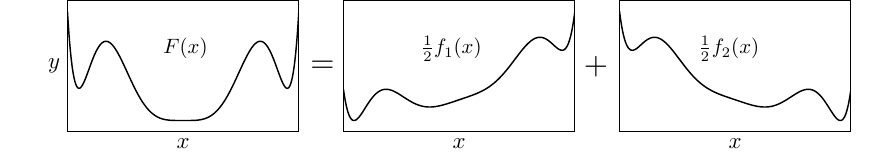}
\end{minipage}
\caption{Visualization of the SGD model problem given by \eqref{F_splitting_ex}--\eqref{f_1f_2} and \eqref{F_gl_min} for values (Top) $\lambda> \lambda_1$, (Middle) $ \lambda_1 < \lambda< \lambda_2$, and (Bottom) $\lambda < \lambda_1$.  When $\lambda > \lambda_2$, the SGD iterates can cross over the barriers of $F$ and there is a unique invariant measure. When  $\lambda_1 < \lambda < \lambda_2$, for any $x_0 \in I$, the SGD iterates get trapped in a sub-optimal local minimum and there are two invariant measures. Lastly, when $\lambda< \lambda_1$ there are three invariant measures. \label{GL_Min_plot}}
\end{figure} 

%The state space $I$ is given by the interval from the smallest critical point of $f_1$ to the largest critical point of $f_2$ (see \cref{LR_GL_min}). 
%\par
From the critical points of $f_1$ and $f_2$ we can construct the left set $L$ and right set $R$ shown in \cref{LR_GL_min}, along with the state space $I$ as the interval from the smallest critical point of $f_1$ to the largest critical point of $f_2$.  The set $T$ is then constructed via \cref{Def:SetsTj} and shown in \cref{Supp_GL_Min}.

%In addition, using \cref{Def:SetsTj} we can use the sets $L$ and $R$ to construct the absorbing set $T$ shown in \cref{Supp_GL_Min}. 

\begin{figure}[htb!]
      \centering
  \includegraphics[scale=.9]{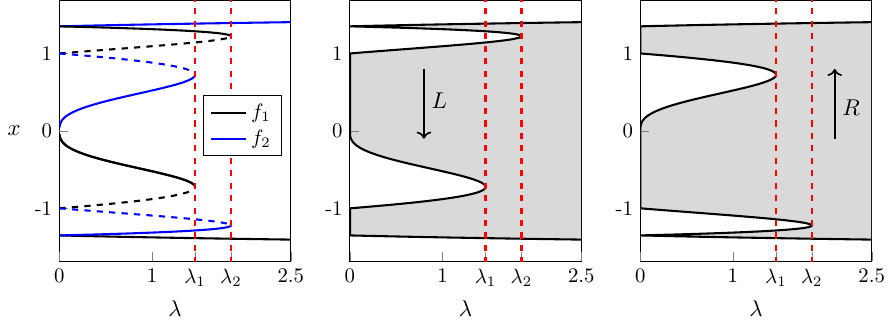}
\caption{Model problem from \eqref{F_gl_min}. Left: the critical points of the functions $f_1$ (black) and $f_2$ (blue) are plotted as a function of $\lambda$.  Solid curves represent local minimum and dashed curve represent local maximum. Middle: for each $\lambda$ the vertical cross section of the filled in region represents the left moving set $L$. Right: for each $\lambda$ the vertical cross section of the filled in region represents the right moving set $R$.  }
 \label{LR_GL_min}
 \end{figure}

 \begin{figure}
      \centering
  \includegraphics[scale=.9]{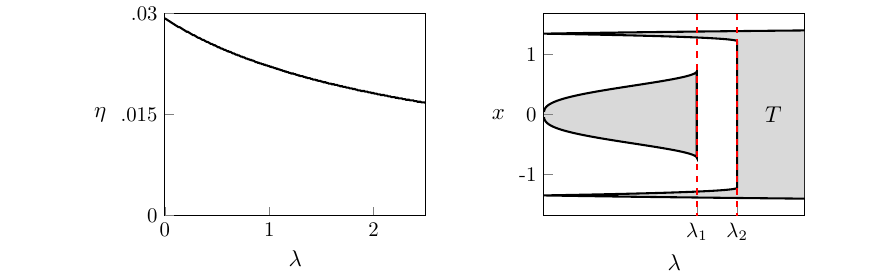}

  \caption{Model problem from \eqref{F_gl_min}. Left: the bound on $\eta$ for the validity of \cref{Main_thm_basin} given in \eqref{eta_bbd_gl_min} is plotted as a function of $\lambda$. Right: for each $\lambda$ the vertical cross section of the filled in region represents the absorbing set $T$. For $\lambda> \lambda_2$ we have one absorbing interval, for $ \lambda_1 <\lambda \leq \lambda_2$ we have two absorbing intervals, and for $\lambda \leq \lambda_1$ we have three absorbing intervals.  \label{Supp_GL_Min}}
 \end{figure}

\par
For all values
\begin{equation}
    \eta < \left( \max \left\{\textrm{Lip} \; f_1^{\prime} ,\textrm{Lip} \; f_2^{\prime} \right\} \right)^{-1}
    \label{eta_bbd_gl_min}
\end{equation}
(see \cref{Supp_GL_Min}) we can apply \cref{Main_thm_basin} to obtain the following results.

\bigskip 
\noindent\fbox{Case 1: $\lambda > \lambda_2$.} The functions $f_1, f_2$ have a unique critical point (see  \cref{LR_GL_min})  and a single absorbing interval $T_1 $ (see \cref{Supp_GL_Min}).
    By \cref{Main_thm_basin} the Markov operator $\mathcal{P}$ has a unique invariant measure $\mu^\star$ supported on $T_1$. \cref{GL_Min_Mu1} compares $\mu^\star$ (computed numerically using Ulam's method \cite{Ulam_method}) to a time histogram of the SGD iterates $X_k$ and diffusion approximation  $\rho^\star$ given by \eqref{Rho_examples} and \eqref{F_gl_min}. 
        
    %(see \cref{GL_Min_Mu1}). In \cref{GL_Min_Mu1} we numerically approximate $\mu^*$ using Ulam's method \cite{Ulam_method} and also by computing the time averages of the iterates $x_k$, which shows that  $\mu^\star$ is approximated well by the diffusion approximation $ \int \rho^\star(x)dx$, where $\rho^\star$ is given in \eqref{Rho_examples} and \eqref{F_gl_min}.   
    \par 
    In addition, from \cref{Main_thm_basin} there exists an $\ell_1>0$ such that for any initial $\mu_0$ supported on $I=T_1$ we have
     \begin{equation*}
        d_F(\mu_{k }, \mu^\star) \leq  \left(1-\frac{1}{2^{\ell_1}} \right)^{\lfloor k/\ell_1 \rfloor} d_F(\mu_0, \mu^\star) \, .
    \end{equation*}

    \bigskip
    \noindent\fbox{Case 2: $\lambda_2 \geq \lambda > \lambda_1$.} The functions $f_1, f_2$ each have three (distinct) critical points (see \cref{GL_Min_plot} and \cref{LR_GL_min}). \cref{LR_GL_min} shows the construction of the sets $L$ and $R$, which leads to the sets $T$ displayed in \cref{Supp_GL_Min}. In this case, there are two absorbing intervals $T_1$ and $T_2$.
    
   By \cref{Main_thm_basin} the Markov operator $\mathcal{P}$ has two invariant measures $\mu_1^\star$ and $\mu_2^\star$ (see \cref{GL_min_mu1_mu2}) supported on $T_1$ and $T_2$ respectively.  Furthermore, any measure $\mu_0$ supported on $I$ converges to a convex combination of $\mu_1^\star$ and $\mu_2^\star$. %, i.e., 
    %\begin{equation}
    %    d _F\Big( \mu_{k } , \,  c \mu_1^\star+ (1-c) \mu_2^\star \Big)  \leq 3 \gamma^{ \lfloor k/\ell \rfloor} 
    %\label{Con_result_Glmin_mu1_mu2}
    %\end{equation}
    %for some $0 \leq c \leq 1$, $0< \gamma < 1$, and $\ell \in \mathbb{N}$.
     \par 
   From \cref{GL_min_mu1_mu2} we see that both $\mu_1^\star$ and $\mu_2^\star$ are supported around the sub-optimal local minimum of $F(x)$ at $x=-1.35$ and $x=1.35$ respectively, and hence both do no agree with the stationary density of  the diffusion approximation given in \eqref{Rho_examples}. In addition, both $T_1$ and $T_2$ do not contain $x=0$ (see \cref{Supp_GL_Min}), which is the global minimum of $F(x)$. Thus, from \cref{Main_thm_basin}  the global minimum $x=0$ is not contained in the support of invariant measures $\mu_1^\star$ and $\mu_2^\star$, and the random iterates $X_k$ do not sample the global minimum as $k \rightarrow \infty$.

    \bigskip 
    \noindent\fbox{Case 3: $\lambda \leq \lambda_1$.} The functions $f_1, f_2$ each have five (distinct) critical points (see \cref{GL_Min_plot} and \cref{LR_GL_min}). In this case, there are three absorbing intervals $T_1$, $T_2$, and $T_3$ (see \cref{Supp_GL_Min}). The set $T_2$ contains the global minimum $x=0$, while $T_1$ and $T_3$ each contain the sub-optimal local minimum $x=-1.35$ and $x=1.35$ respectively. 
    \par 
    Thus, from \cref{Main_thm_basin}  the Markov operator $\mathcal{P}$ has three invariant measures $\mu_1^\star$, $\mu_2^\star$, and $\mu_3^\star$ each shown in \cref{GL_min_mu1_mu2_mu3}.  From \cref{GL_min_mu1_mu2_mu3} we can see that $\mu_1^\star$ and $\mu_2^\star$ do not agree with the diffusion approximation, while $\mu_2^\star$, whose support contains the global minimum of $F(x)$, does agree with the diffusion approximation.

%===============================
% Figures
%===============================

 \begin{figure}[htb!]
      \centering
  \includegraphics[scale=.9]{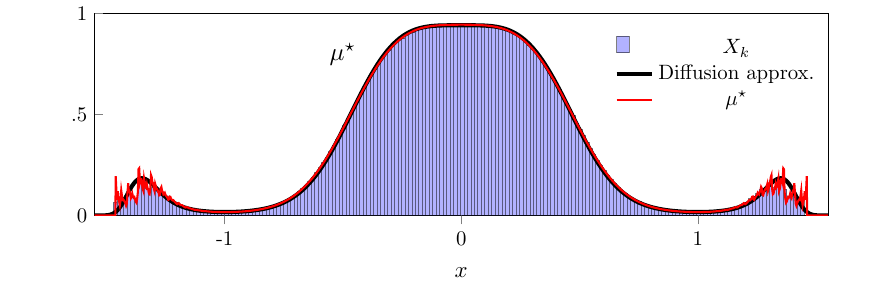}

  \caption{Model problem from \eqref{F_gl_min}. For $\parval=7 > \parval_c$ and $\eta=.007$ (satisfying \eqref{eta_bbd_gl_min}), the probability density function
 of the unique invariant measure  $\mu^\star$ is plotted. Ulam's method \cite{Ulam_method} is used to estimate the exact invariant measure (red). A histogram of the iterates $X_k$ is plotted (blue). The stationary density of the diffusion approximation given in \eqref{Rho_examples} and \eqref{F_gl_min} is plotted (black). }
   \label{GL_Min_Mu1}

 \end{figure}

 \begin{figure}[htb!]
      \centering
 \begin{minipage}[b]{.8\textwidth}
\includegraphics[trim={.7cm 0 0 0},clip,width=\textwidth]{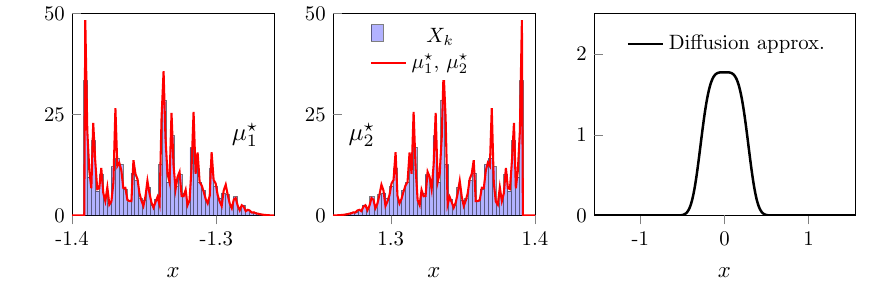}
\end{minipage}
%\hfill
\begin{minipage}[b]{.8\textwidth}
\includegraphics[trim={1cm 0 0 0},clip,width=\textwidth]{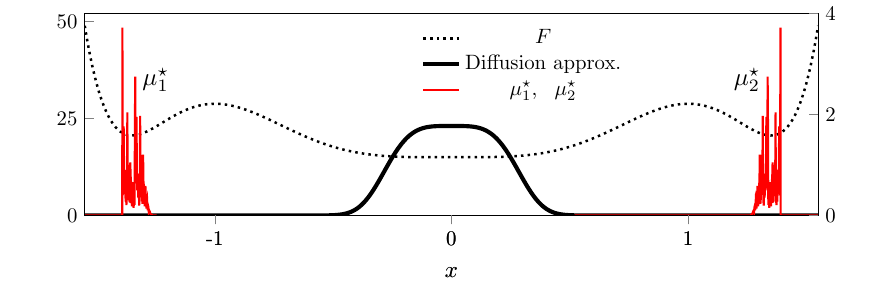}
\end{minipage}

  \caption{Model problem from \eqref{F_gl_min}. Bottom panel shows the relation of the exact invariant measures $\mu_j^{\star}$ ($j=1,2$) (vertical scale units on left) and diffusion approximation (vertical scale units on right) to the objective function (dashed line, arbitrary units).   
  %For $\lambda_1<\parval=1.8 < \parval_2$ and $\eta=.015$ (satisfying \eqref{eta_bbd_gl_min}), the probability density function of $\mu_1^\star$ (left), $\mu_2^\star$ (middle), and the diffusion approximation $\rho^\star$ (right) is plotted. Ulam's method \cite{Ulam_method} is used to estimate $\mu_1^\star$ (red left) and $\mu_2^\star$ (red middle). In addition, a histogram of the iterates $x_k$ is also plotted (blue). The stationary density $\rho^\star$ of the diffusion approximation given in \eqref{Rho_examples} and \eqref{F_gl_min} is also plotted (right).  
  This example demonstrates that the global minimizer of $F$ is not contained inside the support of either invariant measure $\mu_j^{\star}$ ($j=1,2$) which are plotted in red (time-histograms of the SGD iterates $X_k$ are in blue). The supports of $\mu_j^{\star}$ are each on a neighborhood of the sub-optimal minimizers of $F$. In contrast, the diffusion approximation $\rho^\star$ (black) is localized to the neighborhood of the global minima of $F$ and fails to approximate the invariant measures.  The top row plots an enlarged image of each measure.
 }
 \label{GL_min_mu1_mu2}
 \end{figure}

  \begin{figure}[htb!]
      \centering
  \includegraphics[scale=.9]{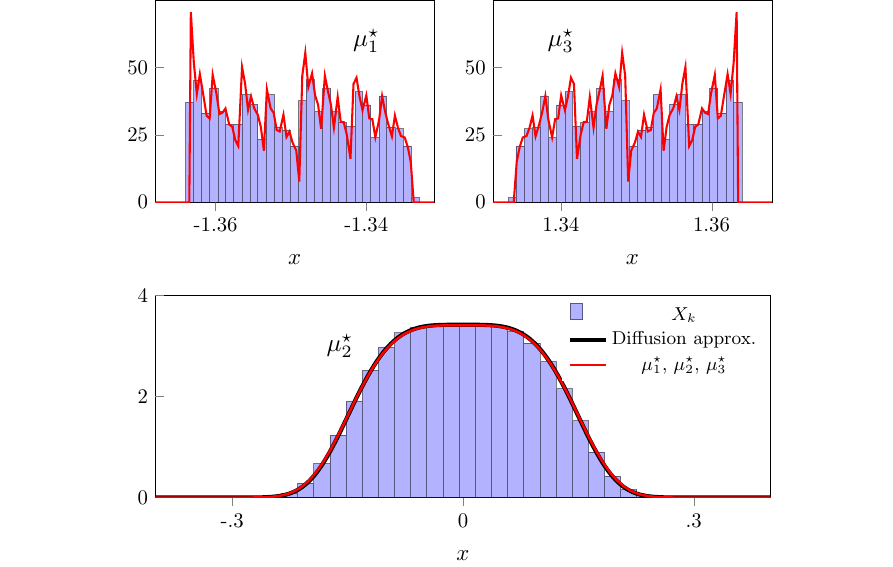}

  \caption{Model problem from \eqref{F_gl_min}. For $\parval=.5 < \parval_1$ and $\eta=.015$ (satisfying \eqref{eta_bbd_gl_min}), the probability density function
 of $\mu_1^\star$ (top left), $\mu_3^\star$ (top right), $\mu_2^\star$ (bottom), and the diffusion approximation (black bottom) is plotted. Ulam's method \cite{Ulam_method} is used to estimate $\mu_1^\star$, $\mu_2^\star$, and $\mu_3^\star$ (red). In addition, a histogram of the iterates $X_k$ is also plotted (blue). Note that $\rho^{\star}$ approximates $\mu_2^{\star}$ but fails to approximate the other two invariant measures. }
 \label{GL_min_mu1_mu2_mu3}
 \end{figure}

%==============================================================================
\newpage
\section{Mathematical Background}
\label{Sec:MathBackground}
%==============================================================================
This section collects the mathematical background used throughout the paper. 

We say that a non-empty set $T \subset \mathbb{R}^\dimR$ is \emph{positive invariant} for the SGD dynamics, if 
\begin{equation}\label{Def:Trapping}
   \varphi_i(T) \subset T \, , \qquad \textrm{for all } 1 \leq i \leq n \, .
\end{equation}
For a general state space, a Borel set $T$ is \emph{absorbing} \cite[Section 4.2.2]{Meyn_Tweedie} if the Markov transition kernel $p(x,T) = 1$ for all $x \in T$. Identity \eqref{Def:Trapping} ensures that if the initial measure $\mu_0$ is supported on $T$, then all successive iterations $\mu_k$ defined by \eqref{Eq:MarkovOperatorDynamics} will be supported on $T$. Hence, for the SGD dynamics, a Borel set $T$ is absorbing if and only if it is positive invariant.

A Borel set $B$ is \emph{uniformly transient} \cite[Chapter 8]{Meyn_Tweedie} if the function 
\begin{align*}
    U(x) := \sum_{n=0}^{\infty}(\mathcal{P}^n \delta_x)(B)\, , 
\end{align*}
is bounded on $B$, i.e., $\sup_{x \in B} U(x) < \infty$.  The function $U(x)$ measures the expected number of times the Markov chain initialized to $X_0 = x \in B$ visits $B$; uniformly transient sets are expected to spend only a finite time in $B$. The inequality \eqref{Eq:Mass_On_B} implies $U(x)$ is bounded by a geometric series, and hence the associated set $B$ is uniformly transient.

{The dual of the Markov operator $\mathcal{P}$ is defined on bounded measurable functions $f$ as
\begin{equation*}
    (\mathcal{P}^{\star} f) (x) := \int_{\mathbb{R}^\dimR} f(y) \, p(x, dy) \, .
\end{equation*}
For iterated functions systems where $p$ is given by \eqref{Def:TransitionKernel}, the  operator $\mathcal{P}^{\star}$ takes the form
\begin{align}\label{Def:Pstar}  
    (\mathcal{P}^{\star} f)(x)= \frac{1}{n} \sum_{i=1}^n f\big(\varphi_{i}(x) \big) \, . 
\end{align}
When $\varphi_i$ is continuous (e.g., as in Assumption~\ref{A1}), $\mathcal{P}^{\star}$ is said to be \emph{weak Feller} as it maps continuous functions into continuous functions.
Writing $\langle f, \mu \rangle := \int_{\mathbb{R}^d} f(x) \, d\mu(x)$ as the pairing between bounded measurable functions and finite measures, then
\begin{align}\label{Eq:AdjointProperty}
    \langle f, \mathcal{P} \mu \rangle = \langle \mathcal{P}^\star f, \mu \rangle \, .
\end{align}
Throughout $\delta_x$ denotes the Dirac measure on $\mathbb{R}^{\dimR}$ and $\delta_{\Tindexd \Tindexd'}$ the Kronecker delta (i.e., $\delta_{\Tindexd \Tindexd'} = 1$ when $\Tindexd = \Tindexd'$ and $0$ otherwise).
}

Let $\dtuple = (\dtuple_1, \, \dtuple_2, \, \ldots, \, \dtuple_\dimR) \in \{-1, +1\}^\dimR$ denote the $\dimR$-tuple of signed unit coefficients, and associate to each of the $2^\dimR$ vectors $\dtuple$ the closed orthont of $\mathbb{R}^\dimR$
\begin{align*}
    \mathbb{R}_{\dtuple}^{\dimR} :=
     \left\{ c_1 \, (\dtuple_1 e_1) + c_2 \, (\dtuple_2 e_2) \hdots + c_\dimR \, (\dtuple_\dimR e_\dimR)  \, : \, c_j \geq 0  \textrm{ for } 1\leq j \leq \dimR \right\} \, .
\end{align*}
Here $e_j$ denotes the usual basis vectors in $\mathbb{R}^\dimR$.  Each orthant $\mathbb{R}_{\dtuple}^{\dimR}$ then defines a partial ordering $\preceq_{{\dtuple}}$ on $\mathbb{R}^\dimR$ given by 
\begin{equation*}
    x \preceq_{{\dtuple}}  y  \qquad \textrm{if and only if} \qquad y-x \in \mathbb{R}_{\dtuple}^{\dimR} \, .
\end{equation*}
For two (non-empty) sets $A, B \subset \mathbb{R}^\dimR$, we write $x \in A \setminus B$ if $x \in A$ and $x \notin B$, and also $A \preceq_{\dtuple} B$ if $a \preceq_{\dtuple} b$ for all $a \in A$ and $b \in B$.

A map $\gamma : S \rightarrow S$ ($S \subset\mathbb{R}^{\dimR}$) is said to be monotone with respect to the cone $\mathbb{R}_{\dtuple}^{\dimR}$ on a set $S$ if for any $x, y \in S$
\begin{equation}\label{Eq:MonotoneProperty}
    x \preceq_{\dtuple} y  \qquad \textrm{implies} \qquad  \gamma(x) \preceq_{\dtuple} \gamma(y) \, .
\end{equation}
Note that if $x \preceq_{\dtuple} y$ then $y \preceq_{-\dtuple} x$. To avoid this redundancy, we take $\alpha_1 = +1$.

The metric used in the convergence result of Bhattacharya and Lee makes use of the following restricted class of sets. Define $\mathcal{A}_{\dtuple}$ to be the family of sublevel sets of continuous monotone maps
\begin{equation*}
   A \in  \mathcal{A}_\dtuple \qquad \textrm{if} \qquad A = \left\{ y \in \mathbb{R}^{\dimR} \; : \; \gamma(y) \preceq_\dtuple c \right\} \, , 
\end{equation*}
for some constant vector $c \in \mathbb{R}^{\dimR}$ and continuous function $\gamma : \mathbb{R}^d \rightarrow \mathbb{R}^d$ that is monotone with respect to $\mathbb{R}_{\dtuple}^{\dimR}$.  When $\alpha = (+1, +1, \ldots, +1)$, the set $\mathcal{A}_{\dtuple}$ includes all (semi-infinite) rectangles of the form $(-\infty, c_1] \times (-\infty, c_2] \times \ldots \times (-\infty, c_\dimR]$, for $c \in \mathbb{R}^{\dimR}$, but also includes other sets as well. 

For a closed and bounded Borel set $I \subset \mathbb{R}^\dimR$ (which, in practice, we will take as a rectangle), and two (Borel) probability measures $\mu, \nu \in \mathscr{P}(I)$, the metric of Bhattacharya and Lee is
\begin{equation}\label{Def:d_BhatLeeMetric}
    d_{\dtuple}(\mu,\nu) := \sup_{A \in \mathcal{A}_{\dtuple}}  |\mu(A \cap I )-\nu(A \cap I ) | \, .
\end{equation}
For closed and bounded $I$, the metric space $(\mathscr{P}(I), d_{\dtuple})$ is complete \cite{Bhattacharya88, BhattacharyaLee1997}.

In dimension $\dimR = 1$, $d_{\dtuple}(\mu,\nu)$ reduces to the Komolgorov distance (which we write as $d_F$): 
\begin{equation*}
   d_F( \mu,\nu) = \| F_{\mu}(x) - F_{\nu}(x) \|_{\infty} \, \qquad \mu, \nu \in \mathscr{P}(I) \quad (I = [a,b]) \, , 
\end{equation*}
where $F_{\mu}(x) := \mu\left( [a,x] \right)$ is the cumulative distributions function (CDF) and
\begin{align*}
    \| f \|_{\infty} = \sup_{x \in [a,b]} |f(x) |\, .
\end{align*}

The metric $d_{\dtuple}$ is stronger than the Wasserstein metric and weak convergence. The converse is not true, e.g., in $\dimR = 1$, $\mu_k = \delta_{\frac{1}{k}}$ converges weakly to $\delta_0$ but not with respect to $d_F$. On the other hand, $d_{\dtuple}$ is weaker than the total variation metric $d_{\rm TV}(\mu,\nu)$ which is defined as \eqref{Def:d_BhatLeeMetric} where the supremum is taken over all sets in the Borel $\sigma$-algebra. 

The metric $d_{\dtuple}$ generalizes naturally to all finite non-negative measures $\mathscr{M}(I)$ with mass $|\mu|$. The following properties also hold:
    \begin{equation}\label{d_F_less1}
        d_{\dtuple}(\mu,\nu) \leq \max\{ |\mu|, |\nu|\} \, , \qquad (\mu, \nu \in \mathscr{M}(I)) \, ;
    \end{equation}     
    and, for all $\mu_1, \mu_2, \nu_1, \nu_2 \in \mathscr{M}(I)$ and $c>0$,
    \begin{align} \label{dF_split11}
                d_{\dtuple}(\mu_1 + \mu_2,\nu_1+ \nu_2) & \leq |\mu_1| + |\nu_1| + d_\dtuple(\mu_2,\nu_2) \, , \\ 
                \label{d_F_tri}
                  d_{\dtuple}(\mu_1 + \mu_2,\nu_1+ \nu_2) & \leq d_\dtuple(\mu_1,\nu_1)    + d_\dtuple(\mu_2,\nu_2) \, , 
      \\ \label{dF_split1}
                d_{\dtuple}( c\, \mu_1  , c \, \nu_2) &=  c \,   d_{\dtuple}(\mu_1,\nu_1) \,.
        \end{align}
    The identities \eqref{d_F_less1}--\eqref{dF_split1} follow directly from the definition of $d_{\dtuple}$ and the triangle inequality.
    
    To represent trajectories taken by iterates of SGD, let
    \begin{align*}
        \vec{\findex} = \begin{pmatrix}
        \findex_1, & \findex_2, & \ldots \, , & \findex_{m-1}, & \findex_m
    \end{pmatrix} \in [n]^m \, ,
    \end{align*}
where each $\findex_k \in [n]$ $(1\leq k \leq m)$ and $|\vec{i}| = m$ is the \emph{path length}. The \emph{path} defined by $\vec{i}$ is the composition
\begin{align*}
    \varphi_{\, \vec{\findex} \, }(x) := \varphi_{\findex_m}\big( \cdots \varphi_2\big( \varphi_1(x) \big) \cdots \big) = \varphi_{\findex_m} \circ \varphi_{\findex_{m-1}} \circ \cdots \circ \varphi_{\findex_1}(x) \, . 
\end{align*}
Two paths are distinct, $\vec{i} \neq \vec{j}$, if they differ in at least one entry.  For two paths $\vec{j}$ and $\vec{i}$ with lengths $|\vec{j}| = m_1$ and $|\vec{i}| = m_2$ we write 
\begin{align*}
    \vec{i} \circ \vec{j} =    
    \begin{pmatrix}
        j_1, & \ldots \; , & j_{m_1}, & i_1 , & \ldots \; , & i_{m_2} 
    \end{pmatrix}
\end{align*}
to be the concatenation of the paths, so that
\begin{align*}    
    \varphi_{ \, \vec{i} \circ \vec{j} \, }(x) := 
    \varphi_{\, \vec{i}\,} \big( \varphi_{\, \vec{j} \, }(x) \big) = 
    \varphi_{\, \vec{i}\,} \circ \varphi_{\, \vec{j} \, }(x) \, . 
\end{align*}

%==============================================================================
\section{Markov Operators for Iterated Function Systems with Monotone Maps}\label{Sec:MarkovCvg}
%==============================================================================
We state here a result of Dubins and Freedman \cite{Dubins66} and the higher dimensional extension by Bhattacharya and Lee \cite{Bhattacharya88, BhattacharyaLee1997}; note that a similar result was independently proven by Hopenhayn and Prescott \cite{Hopenhayn87} (cf. \cite{HopenhaynPrescott92}). The theorems yield sufficient conditions for convergence of Markov operators arising from iterated function systems with monotone maps.  The original theorems are stated in greater generality than we need, and so we state them as they would be applied to the notation and dynamics in \eqref{Eq:SGDIterates}. 

\begin{theorem}[Dubins and Freedman {\cite[Theorem 5.10]{Dubins66}}]
\label{Thm:DB_Path}
   Assume $\varphi_i : [a,b] \rightarrow \mathbb{R}$ are continuous and strictly monotone (increasing or decreasing) on $[a,b]$ and that $[a,b]$ is positive invariant. If there exists distinct pairs $i \neq j$, $1 \leq i, j \leq n$, and a point $x \in [a,b]$ for which 
   \begin{equation}
    \label{Con_DF1}
        \varphi_i\big([a,b]\big) \subset[a,x]  \qquad \textrm{and} \qquad 
        \varphi_j\big([a,b] \big) \subset[x,b] \, ,
    \end{equation}
    then there exists a unique invariant measure $\mu^\star$ to the operator $\mathcal{P}$ given in \eqref{Eq:SGD_Markov_Operators}. 
    
    Furthermore, for any measure $\mu$ supported on $[a,b]$ we have, 
    \begin{equation}\label{Eq:DF_Conv_Rate}
        d_F( \mathcal{P} \mu, \mu^\star) \leq  \left(1-\frac{1}{n} \right) d_F(\mu, \mu^\star) \, ,
    \end{equation}
    where $(1-\frac{1}{n})$ is the geometric rate of convergence. 
\end{theorem}
Dubins and Freedman referred to \eqref{Con_DF1} as a \emph{splitting} condition since $\varphi_{i}$ and $\varphi_j$  map $[a,b]$ into disjoint subintervals --- split by $x$.  If the maps $\varphi_i$ are monotone increasing, we can re-state \eqref{Con_DF1} as 
\begin{equation}
    \label{Path_cond}
        \varphi_i(b) \leq \varphi_j(a)  \qquad (i \neq j) \, , 
\end{equation}
for some distinct pairs $1 \leq i,j \leq n$ (see \cref{Fig_path}). 

\begin{figure}[!ht]
\centering
\begin{tikzpicture}[]
\draw[latex-latex] (-5,0) -- (5,0) ; 
    \draw[] (-4,-.25) node[below] {$a$} -- (-4,.2);
        \draw[] (4,-.25) node[below] {$b$} -- (4,.2);
                \draw[] (1,-.25) node[below] {$\varphi_j(a)$} -- (1,.2);
              \draw[] (-1,-.25) node[below] {$\varphi_i(b)$} -- (-1,.2);

\path[->]
    (-4,0) edge[bend left] node [left] {} (.9,0.2);
    \path[->]
    (4,0) edge[bend right] node [left] {} (-.9,0.2);
\end{tikzpicture}
\caption{Visualization of condition \eqref{Con_DF1} and \eqref{Path_cond}.}
\label{Fig_path}
\end{figure}
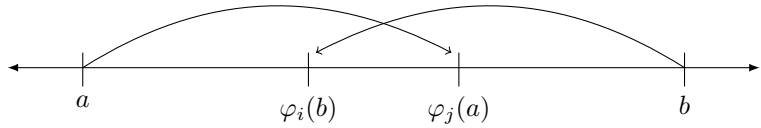

It may be that no pair of $i, j$ satisfy \eqref{Path_cond}. In this case, \cref{Thm:DB_Path} can be applied to powers of the operator $(\mathcal{P})^\ell$ where \eqref{Path_cond} (equivalently \eqref{Con_DF1}) can be replaced with a pair of distinct paths of the same length satisfying 
\begin{equation}
 \label{Path_condition}
  \varphi_{ \, \vec{i} \, }(b) \leq \varphi_{\, \vec{j} \, }(a)   \qquad (\vec{i} \neq \vec{j}) 
  \qquad |\vec{i}| = |\vec{j}| = \ell \, .  
\end{equation}
Inequality~\eqref{Path_condition} is exactly \eqref{Path_cond} with $\{ \varphi_i \; : \; i \in [n]\}$ replaced by $\{ \varphi_{\vec{i}} \; : \; \vec{i} \in [n]^\ell \}$.  Applying \cref{Thm:DB_Path} with condition \eqref{Path_condition} in lieu of \eqref{Con_DF1} modifies the geometric factor appearing in the convergence rate \eqref{Eq:DF_Conv_Rate} to $(1- n^{-\ell})$. 

The result of Bhattacharya and Lee extends \cref{Thm:DB_Path} to dimensions $\dimR > 1$ where the splitting and monotonicity conditions are with respect to a cone. 

\begin{theorem}[Bhattacharya and Lee {\cite[Theorem 2.1 \& Corollary  2.4]{Bhattacharya88}} ]
\label{Bhat_Lee}
    Set $I = [0,1]^{\dimR}$ and let $\dtuple$ be fixed. Suppose $\{ \varphi_i\}_{i=1}^n : I \rightarrow I$ are each continuous and monotone with respect to $\mathbb{R}_{\dtuple}^{\dimR}$.     
    If there are two paths $\vec{p}_1$ and $\vec{p}_2$ of length $\ell$ satisfying 
    \begin{equation}\label{Eq:BL_PathCond}
        \varphi_{\vec{p}_1}( I ) \preceq_{\dtuple} x_0 \qquad \textrm{and} \qquad 
        x_0  \preceq_\dtuple \varphi_{\vec{p}_1}( I ) 
    \end{equation}
  for some $x_0 \in I$, then the Markov operator $\mathcal{P}$ has exactly one invariant measure $\mu^\star$ with support contained in $I$. In addition, for any $\mu$ supported on $I$ 
  \begin{equation}\label{Eq:BL_ErrCvg}
      d_\dtuple ( {\mathcal{P}}^{k} \mu, \mu^\star) \leq \left( 1 - \frac{1}{n^{\ell}} \right)^{\lfloor k/\ell \rfloor} \, \qquad \quad k > 0 \, .
  \end{equation}  
\end{theorem}

There are a few minor differences from the exact theorem statement \cite[Theorem 2.1 \& Corollary  2.4]{Bhattacharya88} and the version we state in \cref{Bhat_Lee}.  In \cite{Bhattacharya88}, \eqref{Eq:BL_ErrCvg} is stated with $\mu$ replaced by $\delta_x$ (the upper bound in \eqref{Eq:BL_ErrCvg} being uniform in $x \in I$). This is equivalent to \eqref{Eq:BL_ErrCvg} as written above. 

The formulation in \cite{Bhattacharya88} also allows for the maps $\varphi_i$ to be drawn from an infinite index set.  In this case, the convergence rate \eqref{Eq:BL_ErrCvg} (or \eqref{Eq:DF_Conv_Rate}) is given in terms of the probability that the event \eqref{Eq:BL_PathCond} holds. If $\kappa$ pairs of distinct paths satisfy \eqref{Eq:BL_ErrCvg}, then \cite{Bhattacharya88} yields the stronger geometric rate of $1 - \kappa/n^{\ell}$ in \eqref{Eq:BL_ErrCvg}.  

%In addition, the original convergence rate \eqref{Eq:BL_ErrCvg} (or \eqref{Eq:DF_Conv_Rate}) appearing in \cite{Bhattacharya88} is stated in terms of the probability that the event \eqref{Eq:BL_PathCond} holds.  In our setting, the bound in \eqref{Eq:BL_ErrCvg} arises since \eqref{Eq:BL_PathCond} implies a probability of at least $1/n^{\ell}$ for one pair of paths $\vec{p}_1$ and $\vec{p}_2$.  If $\kappa$ pairs of distinct paths satisfy \eqref{Eq:BL_ErrCvg}, then \cite{Bhattacharya88} yields the stronger geometric rate of $1 - \kappa/n^{\ell}$ in \eqref{Eq:BL_ErrCvg}.  Even more generally, the formulation in \cite{Bhattacharya88} allows for settings where the maps $\varphi_i$ may be drawn from an infinite index set.  

%Replacing the right-hand side of \eqref{Eq:BL_ErrCvg} with sharper bounds (e.g., obtained from including the combinatorial factor $\kappa$) will yield tighter bounds on the convergence rates of the Markov chain in \cref{Main_thm_basin}.  
Lastly, \cite[Theorem 2.1 \& Corollary  2.4]{Bhattacharya88} is stated only for $\alpha$ being the positive orthont. However, the theorem extends trivially to the version stated in \cref{Bhat_Lee} defined over any cone $\alpha$ by the change of variables: $\varphi(x) \rightarrow \Lambda \varphi( \Lambda x)$ where $\Lambda = \textrm{diag}(\alpha)$.

%================================================================================    
\section{A Few Lemmas for One Dimensional SGD}\label{Sec:Lemmas1d_SGD}
%================================================================================  
Throughout this section we will assume that $d=1$ and build up a series of results for SGD. For notational convenience we refrain from writing superscripts on the functions $f_i$, the maps $\varphi_i$, and the sets $T_\Tindexd$, e.g., we replace
\begin{align*}
    f_i^{(1)} \rightarrow f_i   \qquad
    \varphi_i^{(1)} \rightarrow \varphi_i \quad (1 \leq \findex \leq n)  \qquad \textrm{and} \qquad 
    T_\Tindexd^{(1)} \rightarrow T_{\Tindex} \quad (1 \leq \Tindex \leq \numT) \, .
\end{align*}
Thus, in assumptions \ref{A1}--\ref{A5} each $ f_i^{(1)}$ is simply $f_i$. 

%================================================================================    
\subsection{Properties of the Sets \texorpdfstring{$L$}{L}, \texorpdfstring{$R$}{L}, and \texorpdfstring{$T_\Tindex$}{Lg}}\label{Sec:PropertiesLRT}
%================================================================================
Here our goal is to collect and prove basic properties of the sets $L$, $R$ and $T_\Tindex$ (see \eqref{L_def}, \eqref{R_def}, and \cref{Def:SetsTj}). 

\begin{proposition}[Basic properties of $L$ and $R$]\label{Prop:BasicLR} 
Let $d=1$. Given assumptions \ref{A1}--\ref{A4} and $\leftint, \rightint$ defined by $I$ in \eqref{I_state_space}, then sets $L$ and $R$ defined in \eqref{L_def} and \eqref{R_def} are finite unions of open intervals with $\partial L, \partial R \subset I$ and
\begin{align}\label{Eq:SetInclusion}
    (-\infty, a) \subset R \setminus L \qquad \textrm{and} \qquad (b, \infty) \subset L \setminus R \, .
\end{align}

If in addition \ref{A5} holds, then 
\begin{align}\label{Eq:SetCover}
    \partial L \cap \partial R = \phi \qquad \textrm{and} \qquad L \cup R = \mathbb{R} \, . 
\end{align}
\end{proposition}

\begin{proof}
    The sets $L$ and $R$ are finite unions of open intervals since each $f'_i$ is continuous with a finite number of roots. The boundaries $\partial L, \partial R$ are contained in the set of critical points $\mathcal{C}$ of the $f^{\prime}_i$'s and confined to $I$. 
    
    The assumptions \ref{A1}--\ref{A3}, with $d=1$, imply \eqref{Eq:SetInclusion} since
    \begin{align*}
        f^\prime_i(x) < 0 \quad \textrm{if} \quad x < a \quad \textrm{for all} \quad 1\leq i \leq n \, ,   
    \end{align*}
    and similarly $f^\prime_i(x) > 0$ for $x > b$. Lastly, $x \in \partial L \cap \partial R$ implies $f_i'(x) = 0$ for all $i$ violating \ref{A5}. 
    Since for every $x \in \mathbb{R}$ from \ref{A5}, either $f_i^\prime(x) < 0$ or $f^\prime_i(x) > 0$ for some $i$. Thus, if $x \notin L$, then $x \in R$ and vice versa.     
\end{proof}

The next proposition in this section establishes basic properties of the sets $T_\Tindex$. 

\begin{proposition}[Properties of $T_\Tindex$]\label{Prop:FiniteTs} Let $d=1$. Under assumptions \ref{A1}--\ref{A5}, the sets $T_\Tindex = [l_\Tindex, r_\Tindex]$ in \cref{Def:SetsTj} satisfy the following:
\begin{enumerate}
    \item [(a)] There is at least one $T_\Tindex$, i.e., $\numT \geq 1$.
    \item [(b)] The endpoints satisfy
        \begin{align}\label{Eq:EndpointProp}
            r_\Tindex \in L \, , \quad r_\Tindex \notin R \qquad \textrm{and} \qquad l_\Tindex \in R \, ,  \quad l_\Tindex \notin L \qquad \textrm{for all} \quad 1 \leq \Tindex \leq \numT \, . 
        \end{align}
    \item [(c)] Each $T_\Tindex$ contains at least one local minimum of $F$. In particular, $\numT \leq M_F$ where $M_F$ is the number of local minima of $F$.
    \item [(d)] The sets $T_\Tindex$ are pairwise disjoint and contained in $I$.
\end{enumerate}
\end{proposition}

\begin{proof}    
\medskip
\noindent 
(a) Through direct construction, we show $\numT \geq 1$.  From identity \eqref{Eq:SetInclusion} in \cref{Prop:BasicLR}, the set $L$ contains an interval of the form $(l, \infty)$ where $l \in \partial L$ is finite (in fact $l \leq b$).  Since equation \eqref{Eq:SetCover} implies $R$ and $L$ cover $\mathbb{R}$, $R$ contains an interval of the form $(\bar{r}, r)$ with $\bar{r} < l < r$, where $r \in \partial R$ is finite. Thus $(l, r)$ satisfies Definition~\ref{Def:SetsTj}.

\medskip
\noindent 
(b) Since $L$ and $R$ are open (see \cref{Prop:BasicLR}) they do not contain their boundary points, hence $r_m \notin R$ and $l_m \notin L$. Since $L \cup R = \mathbb{R}$ condition \eqref{Eq:EndpointProp} holds.

\medskip
\noindent 
(c) We show that $F^\prime (l_{\Tindex}) < 0$ and $F^\prime(r_{\Tindex}) > 0$. This implies that $\textrm{argmin}_{x \in T_{\Tindex}} F(x)$ lies in the interior of $T_{\Tindex}$, and hence must be a local minimizer. 

Since $l_{\Tindex} \notin L$ and $r_{\Tindex} \notin R$ from part (b), we have
\begin{align}\label{Eq:fprimeIneq} 
    f_i^\prime(l_{\Tindex})  \leq  0  \qquad \textrm{and} \qquad f_i^\prime(r_{\Tindex})  \geq  0
    \qquad 
    \textrm{for all} \; 1 \leq i \leq n \, ,
\end{align}
where each of the inequalities in \eqref{Eq:fprimeIneq} are strict for at least one $i$ due to Assumption~\ref{A5}. The sign of $F^\prime$ at $l_\Tindex$ and $r_\Tindex$ then follows from the definition of $F$. 

\medskip
\noindent 
(d) First note the endpoints $l_\Tindex, r_\Tindex$ are confined to the critical points of the $f^\prime_\findex$'s and hence must lie in $I$.  To show the sets $T_\Tindex$ are pairwise disjoint, suppose by contradiction that $T_\Tindex \neq T_j$ and $T_\Tindex \cap T_j \neq \phi$. Since the sets are distinct closed intervals, an endpoint of one interval must lie in the other and also differ from the same endpoint, i.e., without loss of generality the right endpoint of $T_\Tindex$ must differ from the right endpoint of $T_j$ and also intersect $T_j$ so that
\begin{align*}
    r_\Tindex \in [l_j, r_j)  \, .     
\end{align*}
The definition of $T_j$ implies $(l_j, r_j) \subset R$ while part (b) implies $l_j \in R$. Thus, $r_\Tindex \in [l_j, r_j) \subset R$, however this contradicts (b).
\end{proof}

%================================================================================    
\subsection{Properties Related to the SGD Dynamics}\label{Subsec:Properties1D_dynamics}
%================================================================================
The next proposition proves the fact that intervals of $L$ and $R$ bound the regions for which SGD can move to the left and right respectively.
\begin{proposition}[Dynamics related to $L$ and $R$]\label{Prop:LeftRight}
  Let $d=1$.  Assume \ref{A1}--\ref{A3} hold with $L$ and $R$ defined in \eqref{L_def}--\eqref{R_def}, and let $0 < \eta < 1/\LipK$.  Then the maps $\varphi_{i}$ are monotone increasing on $I$, i.e.,
    \begin{align*}
        \varphi_i(x) < \varphi_i(y) \qquad \textrm{for all} \quad x<y \in I \quad 
        \textrm{and} \quad 1 \leq i \leq n \, . 
    \end{align*}
    
    Furthermore, if $(l, r] \subset L \cap I$ with $l \in \partial L$, there exists paths that map $r$ arbitrarily close to $l$ (but no further):  
        \begin{align}\label{Eq:LiminfL}
            \inf_{\vec{p} \in Q} \varphi_{\vec{p} \, }(r) = l \, .
        \end{align}
        Analogously, $[l, r) \subset R \cap I$ with $r \in \partial R$ then 
        \begin{align}\label{Eq:LimsupR}
            \sup_{\vec{p} \in Q} \varphi_{\vec{p} \, }(l) = r \, .
        \end{align}
        Here $Q$ is the set of all paths of arbitrary length.
\end{proposition}
 
\begin{proof} The monotonicity of $\varphi_i$ follows from
\begin{align*}
    \varphi_i(x) < \varphi_i(y) \qquad \Longleftrightarrow \qquad 
    f^{\prime}_i(y) - f^{\prime}_i(x) < \eta^{-1} (y - x) \, , 
\end{align*}
combined with the Lipschitz bound on $f^{\prime}_i$ in \ref{A4} and $\eta < 1/\LipK$.

We prove \eqref{Eq:LiminfL} as \eqref{Eq:LimsupR} follows by an identical argument. Since $l \notin L$, we have 
    \begin{align*}
        \varphi_{i}(l) \geq l \qquad \textrm{for all} \quad 1 \leq i \leq n \, .
    \end{align*}
    Combining this with the fact that the maps $\varphi_{\vec{p}}$ are monotone on $(l, r]$ (which is in $I$) implies
    \begin{align*}
        \alpha := \inf_{\vec{p} \in Q} \varphi_{\vec{p} \,}(r) \geq \inf_{\vec{p} \in Q} \varphi_{\vec{p} \, }(l) \geq l \,. 
    \end{align*}
    Suppose now by contradiction that $\alpha > l$. Introduce the function
    \begin{align*}
        \Delta(x) := \max_{1 \leq i \leq n} f^\prime_i(x) \, .
    \end{align*}
    The function $\Delta(x)$ is continuous (since it is the pointwise maximum of a finite collection of continuous functions) and strictly bounded by $\Delta(x) >0$ on $[\alpha, r]$ (since the interval is contained in $L$). By compactness, $\Delta(x)$ achieves its minimum $\Delta_0$ on $[\alpha, r]$ and satisfies 
    \begin{align*}
        \Delta(x) \geq \Delta_0 > 0 \qquad \textrm{for} \qquad x \in [\alpha, r] \, .     
    \end{align*}    
    However, this yields a contraction by the definition of $\alpha$: For every $x \in [\alpha, r]$ there is a map $\varphi_i$ satisfying
    \begin{align*}
        \varphi_i(x) \leq x - \eta \Delta_0 \, ,        
    \end{align*}
    which implies there is a path $\vec{p}$ of finite length for which $\varphi_{\vec{p} \, }(r) < \alpha$.
\end{proof}

We show next that the sets $T_\Tindex$ are positive invariant, or equivalently, absorbing.
\begin{proposition}[$T_\Tindex$ are positive invariant]\label{Prop:1d_Tj_PosInv}
    Let $d=1$.  Assume \ref{A1}--\ref{A5} hold and $0 < \eta < 1/\LipK$.  Then $I$ and each $T_\Tindex$ ($1 \leq \Tindex \leq \numT$) is positive invariant. 
\end{proposition}
\begin{proof}
Write $T_\Tindex = [l_\Tindex, r_\Tindex]$. Combining the facts that $l_\Tindex \notin  L$ and $r_\Tindex \notin  R$ (by \cref{Prop:FiniteTs}), $\varphi_i$ is monotone on $T_\Tindex$ (by \cref{Prop:LeftRight}), and \eqref{not_inL} and \eqref{not_inR} yields
\begin{align*}    
    l_\Tindex \leq \varphi_i(l_\Tindex) \leq \varphi_i(x) \leq \varphi_i(r_\Tindex) \leq r_\Tindex     
    \qquad \textrm{for all} \quad 1 \leq i \leq n\, \quad \textrm{and} \quad x \in T_\Tindex \, . 
\end{align*}
Hence,
\begin{equation}
\label{T_j_invar_d1}
    \varphi_i(T_\Tindex) \subset T_\Tindex
\end{equation}
holds for all $i$ and $\Tindex$. By the same argument $I=[a,b]$ is also positive invariant.
\end{proof}

We conclude this section with a proof that there is a fixed path length for which every point $x \in I$ outside $T$ can be mapped into $T$.  This will provide the basis for establishing that the set $B$ is uniformly transient in the main result. 
\begin{lemma}[Uniform path length in one dimension]\label{Lem:1d_Pathlength}
    Let $d=1$.  Assume \ref{A1}--\ref{A5} and $0 < \eta < 1/\LipK$ hold and define $T$ as in \eqref{Eq:T_Union}.  Then there exists a uniform path length $\ell_0$ such that for every $x \in I$ there is a path $\vec{p}$ of length $\ell_0$ satisfying $\varphi_{\vec{p}\,}(x) \in \mathrm{int}\, T$, where $\mathrm{int}\, T$ is the interior of $T$.
\end{lemma}

\begin{proof}
Step 1: We first show that for every $x \in I$ there exists a path $\vec{t}_x$, whose length may depend on $x$, such that
\begin{equation}\label{Pathx}
    \varphi_{\, \vec{t}_x \, }(x) \in \inter T \, .
\end{equation}
 
To prove \eqref{Pathx} we show for each $x \in I$, there exists a $T_\Tindex = [l_\Tindex, r_\Tindex]$ satisfying \cref{Def:SetsTj} such that either
\begin{align}\label{Eq:SetIdenity}
    [x, r_\Tindex) \subset R \qquad \textrm{or} \qquad (l_\Tindex, x] \subset L \, .
\end{align}
Condition \eqref{Eq:SetIdenity} combined with \eqref{Eq:LiminfL}--\eqref{Eq:LimsupR} from \cref{Prop:LeftRight} imply \eqref{Pathx}.

The remaining proof of \eqref{Eq:SetIdenity} for Step 1 is visualized in \cref{Fig:ProofPath1d}.  Assume without loss of generality that $x \in (\dvar_1, \dvar_2) \subset L$ where $\dvar_1$ and $\dvar_2$ defines the largest open sub-interval in $L$ containing $x$ (i.e., both $\dvar_1, \dvar_2$ lie on $\partial L$). If $x \notin L$ an identical assumption may be made regarding $x \in R$. 

If $\dvar_1 = l_\Tindex$ for some $\Tindex$, then we are done (see \cref{Fig:ProofPath1d}(a)). 

If $\dvar_1 \neq l_\Tindex$ (for every $\Tindex$), then $\dvar_1 \in R$ by \eqref{Eq:EndpointProp} (see \cref{Fig:ProofPath1d}(b)). Let $(\dvar_1, \dvar_3) \subset R$ be the largest sub-interval of $R$ with $\dvar_3 \in \partial R$. Then $\dvar_2 < \dvar_3$, otherwise $[\dvar_1, \dvar_3]$ would define a $T_\Tindex$ and \eqref{Eq:SetIdenity} would hold.  

By construction, the interval $[x, \dvar_3) \subset R$. We now claim that $\dvar_3 = r_\Tindex$ for some $\Tindex$, in which case we are done. The reason is that since $R \cup L$ covers $I$ (e.g., \eqref{Eq:SetCover}) with $\dvar_3 \in \partial R$, so there must be an interval of the form $(\dvar_4, \dvar_3] \subset L$ with $\dvar_4 \in \partial L$. Lastly, $[\dvar_4, \dvar_3]$ satisfies \cref{Def:SetsTj}, since $\beta_3 >\beta_2$ and $\beta_2 \notin  L$ implies $\beta_4 \geq \beta_2>\beta_1$, which implies that $(\beta_4,\beta_3) \subset R$. 

\begin{figure}
      \centering
       \includegraphics[width=0.95\textwidth]{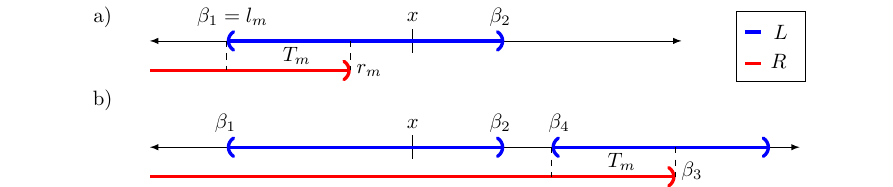}%scale=1.03       
       \caption{Visualization of two sub-cases for the proof in Step 1 of \cref{Lem:1d_Pathlength}.}
        \label{Fig:ProofPath1d}       
\end{figure}

\bigskip 

Step 2: To complete the proof of the Lemma, it is sufficient to show that there exists an $\ell = \ell(\eta)$ for which $I \subset U_{\ell}$ where 
\begin{equation}\label{Def:U_k}    
    U_k := 
    \Big \{ x \in \mathbb{R} \; : \; 
    \varphi_{\, \vec{t} \, }(x) \in \inter T   
    \textrm{ and } 
    |\, \vec{t} \, | = k \Big \} \, ,
\end{equation}
is the set of points that map into the interior of $T$ after $k$ steps.  Let 
\begin{equation*}
    U= \bigcup_{k=0}^\infty U_k \, , 
\end{equation*}
be the set of all points that can reach the interior of $T$ (with an arbitrary path length). Note that $U_k$ is equivalently  
\begin{equation*}
    U_k = \left \{ \varphi_{\, \vec{t} \, }^{-1}( \inter T) \;\; : \;\; |\, \vec{t} \,| = k \right\} \, .
\end{equation*}
Since $\varphi_j$ is continuous and $\inter T$ is open, $U_k$ and hence $U$ is open. By \eqref{Pathx}, the collection $\{U_k \; : \; k \in \mathbb{Z} \}$ is an open cover of $I$ and by compactness has a finite sub-cover, i.e., 
\begin{equation*}
    I \subset \bigcup_{k=0}^{\ell} U_k \, ,    
\end{equation*}
for some $\ell$. Since $\varphi_\findex$ is (strictly) monotone on $I$ when $\eta < 1/\LipK$, and $T$ is positive invariant, we have that each $\varphi_\findex$ maps open subsets of $T$ into open subsets of $T$, whence $\varphi_\findex( \inter T) \subset \inter T$ for every $1 \leq \findex \leq n$. Consequently, the sets $I \cap U_k \subset I \cap U_{k+1}$ are nested, and hence $I \subset U_\ell$. 
\end{proof}

%================================================================================    
\section{Proof of the Main Result}\label{Sec:ProofsMainResult}
%================================================================================
Building on the one dimensional results, we provide the proofs for each part of the main result.

%================================================================================
\subsection{Main Result Proof of Part \ref{Main_a}} \label{Subsec:ProofMainResult_Parta}
%================================================================================

The proof of \cref{Main_thm_basin}\ref{Main_a} makes use of the following lemma --- which is a general statement regarding uniformly transient sets for Markov operators arising from iterated function systems. The lemma makes no assumptions on the regularity or monotonicity of the maps $\varphi_i$. 
\begin{lemma}[Uniformly transient sets for iterated function systems]\label{Lem:Mass_onT}
    Consider the dynamics \eqref{Eq:SGD_Maps} for a family of maps $\varphi_i : I \rightarrow I$ ($1 \leq i \leq n$) where $I \subset \mathbb{R}^{\dimR}$ is non-empty and $\mathcal{P}$ is defined in \eqref{Eq:SGD_Markov_Operators}.  Suppose the set $T \subset I$ ($B := I \setminus T)$ is positive invariant and for each $x \in B$, there exists a $\varphi_i(x) \in T$. 
    
    Then for any finite measure $\mu \in \mathscr{M}(I)$ %supported on $I$,
    \begin{align*}
        (\mathcal{P} \mu)(B) \leq \left( 1 - \frac{1}{n} \right) \mu(B) \, .
    \end{align*}
\end{lemma}
The intuition behind \cref{Lem:Mass_onT} is that a fraction of the mass $\mu(B)$ travels into $T$ every iteration.  The proof is provided here for completion. 

\begin{proof}
        For each $i = 1, \ldots, n$ let $B_i := \{ x\in B \; : \; \varphi_i(x) \in T\}$, i.e., $B_i \subset B$ with $\varphi_i(B_i) \subset T$.  The sets $B_i$ then satisfy the following identities:
        \begin{align}\label{Propt_A}
            \varphi_i^{-1}(B) \cap I \, &=\, B \setminus B_i \, ,
        \end{align}
        holds since $T$ and $I$ are positive invariant, and
        \begin{align}\label{Propt_B}
            \bigcup_{i = 1}^n  B_i  &= B \, ,
        \end{align}
        holds since for each $x \in B$ one of the maps $\varphi_i(x) \in T$. 

        By direct calculation, since $\mu$ is supported on $I$ we then have
        \begin{alignat*}{2}\label{PM_mu_T}
            (\mathcal{P} \mu)(B) = 
        \frac{1}{n} \sum_{i=1}^n \mu\left( \varphi_{\, i \, }^{-1}(B) \cap I \right) 
        &= \mu(B) - \frac{1}{n} \sum_{i=1}^n \mu(B_i) \, , & \quad &(\textrm{Via} \; \eqref{Propt_A}) \, , \\
        &\leq \mu(B) - \frac{1}{n} \mu\left( \cup_{i=1}^n B_i\right) \, . &  & 
\end{alignat*}
    Substituting \eqref{Propt_B} into the last line yields the desired result.
\end{proof}

\begin{proof}[Proof of Main Result \cref{Main_thm_basin}(a)] %\ref{Main_thm_Ti_invar}] 
\bigskip 

\noindent \ref{Main_a_i} For each $1 \leq j \leq d$ the family of functions $\{f_{i}^{(j)} \}_{i=1}^n$ satisfies \ref{A1}--\ref{A5}.  \cref{Prop:FiniteTs}(a) implies for each $j$, $\M_j \geq 1$ so that $\numT \geq 1$ and $T$ is non-empty.  Secondly, \cref{Prop:FiniteTs}(c) implies that for each $\Tindex = 1, \ldots, \M_j$ at least one local minima of 
\begin{align*}
    \wmac^{(j)}(x_j) := \sum_{\findex=1}^{n} f_i^{(j)}(x_j) \, ,
\end{align*}
is contained in each $T_\Tindex^{(j)}$. Since we can write $F(x) = \sum_{j=1}^{\dimR} \wmac^{(j)}(x_j)$ and $T_{\Tindexd}$ is a rectangle of the form \eqref{Eq:DefT_Rd}, it follows that each $T_{\Tindexd}$ ($\Tindexd \in \mathcal{\M}$) contains a local minima of $F$. 

Thirdly, \cref{Prop:1d_Tj_PosInv} implies that $T_\Tindex^{(j)}$ ($1 \leq \Tindex \leq \M_j$) is positive invariant with respect to the maps $\{\varphi_i^{(j)} \}_{i=1}^n$. Combining this fact with the form \eqref{Eq:SepVarphi} of $\varphi_i$ for separable functions implies that both $I$ and $T_{\Tindexd}$ ($\Tindexd \in \mathcal{\M}$) are positive invariant with respect to the $\dimR$ dimensional maps $\{\varphi_i\}_{i=1}^n$.

Additional positive invariant sets can also be constructed via intersections and unions, for instance, rectangles of the form $I^{(1)} \times \ldots \times I^{(j-1)} \times T^{(j)} \times I^{(j+1)} \times \ldots \times I^{(\dimR)}$ are also positive invariant.  This will be used in the proof of Part~(ii). 

\bigskip
\noindent \ref{Main_a_ii} Apply \cref{Lem:1d_Pathlength} separately to each dimension. Then there exists an $\ell^{(j)}$ such that for every $x_j \in I^{(j)}$ there exists a path $\vec{t}_j$ (depending on $x_j$) of length $\ell^{(j)}$ satisfying  $\varphi_{\vec{t}_j}^{(j)}(x_j) \in   T^{(j)}$.  Now let 
   \begin{equation}\label{Def:Ell}
       \ell_0 =  \sum_{j=1}^d \ell^{(j)} \, . 
   \end{equation}
   Then for every $x = (x_1, \, x_2, \, \ldots, \, x_{\dimR}) \in I$ define the composition path $\vec{p} = \vec{p}_1 \circ \vec{p}_2 \circ \cdots \circ \vec{p}_\dimR$ as follows: Set
   $\vec{p}_{\dimR}$ (of length $\ell^{(\dimR)}$) so that
  { \begin{align}\label{Eq:Mapping1}
       \varphi_{\vec{p}_{\dimR}}(x) \subset I^{(1)} \times \ldots \times I^{(\dimR-1)} \times \  T^{(\dimR)} \, . 
   \end{align}}
   Pick $\vec{p}_{\dimR-1}$ with length $\ell^{(\dimR - 1)}$ to map the $\dimR -1$ component of $\varphi_{\vec{p}_{\dimR}}(x)$ into $    T^{(\dimR - 1)}$. The positive invariance of the set in the right hand side of \eqref{Eq:Mapping1} yields:
    \begin{align*}
       \varphi_{\vec{p}_{\dimR-1}\,} \circ \varphi_{\vec{p}_{\dimR}}(x) \subset I^{(1)} \times \ldots \times 
       I^{(\dimR-2)} \times  T^{(\dimR-1)} \times T^{(\dimR)} \, . 
   \end{align*}
    Continuing to define $\varphi_{\vec{p}_j}$ with length $\ell^{(j)}$ to map the $j$th component of $\varphi_{\vec{p}_{j+1}} \circ \cdots \circ \varphi_{\vec{p}_{\dimR}}(x)$ into $   T^{(j)}$, one obtains
    \begin{equation*}
        \varphi_{\vec{p}}(x) \subset   T^{(1)} \times  T^{(2)} \times \cdots \times   T^{(\dimR)} =   T  \, .
    \end{equation*}    
    
\bigskip
\noindent \ref{Main_a_iii} Take $\ell_0$ as defined in \eqref{Def:Ell} and apply \cref{Lem:Mass_onT} to the family of $n^{\ell_0}$ maps $\{ \varphi_{\vec{j}}  :  |\vec{j}| = \ell_0 \}$.   
\end{proof}
 %================================================================================
\subsection{Main Result Proof of Part \ref{Main_b}}
%================================================================================

    \cref{Main_thm_basin}\ref{Main_b} follows directly from \cref{Bhat_Lee}, where the key technical statement to prove is that the SGD dynamics satisfy the splitting condition \eqref{Eq:BL_PathCond}.  
    
    To prove \eqref{Eq:BL_PathCond}, we proceed by induction on the dimension.  In particular, the idea is to consider the restriction of the SGD dynamics (for separable functions) into the first $j$ coordinates. We then show, that if \eqref{Eq:BL_PathCond} holds for the first $j$ coordinates, then \eqref{Eq:BL_PathCond} holds for the first $j+1$ coordinates as well.  
    
    There is a subtlety in the proof in that the argument relies on the dynamics \eqref{Eq:SGDIterates} to construct $\alpha$ defining the orthant $\mathbb{R}_{\dtuple}^{d}$.  The following example demonstrates that \eqref{Eq:BL_PathCond} does not necessarily hold for every choice of $\alpha$, but rather does hold for at least one $\alpha$.    
    \begin{example}
        Set $\dimR = 2$, $f_1(x_1,x_2) = (x_1)^2 + (x_2-1)^2$ and $f_2(x_1, x_2) = (x_1-1)^2 + (x_2)^2$ so that
        \begin{align*}
            f_1^{(1)}(x_1) = x_1^2 \qquad f_1^{(2)}(x_2) = (x_2 - 1)^2
            \qquad 
            f_2^{(1)}(x_1) = (x_1-1)^2 \qquad f_2^{(2)}(x_2) = (x_2)^2 \, . 
        \end{align*}
        In this example, there is only a single set $T = I = [0,1]^2$ as defined in \eqref{Eq:DefT_Rd}. Provided $\eta < 1/\LipK$, the following subsets of $T$ are positive invariant: (i) the line segment $y = 1-x$; (ii) the region $y > 1-x$; and the region (iii) $y < 1-x$.  As a result, the condition \eqref{Eq:BL_PathCond} with respect to $\alpha = (+1, +1)$ is never satisfied (for any pair of paths of any length) since mappings of $(0,0)$ and $(1,1)$ will stay separated by the line segment $y = 1-x$; no pair of maps will flip the ordering of $(0,0)$ and $(1,1)$.         
        
        The condition \eqref{Eq:BL_PathCond} is however readily verified for $\alpha = (+1, -1)$. Despite the fact that $T$ has three positive invariant sets, the main result implies any initial measure $\mu_0 \in \mathscr{P}(I)$ converges to a unique invariant measure (which in this case is supported on the line segment $y=1-x$ ($0\leq x \leq 1$).
    \end{example}
   
    \begin{proof}[Proof of Main Result \cref{Main_thm_basin}(b)] It is sufficient to show that $\{\varphi_i\}_{i=1}^n$ and the set $T_{\Tindexd}$ satisfy the hypothesis in \cref{Bhat_Lee} (or equivalently \cref{Thm:DB_Path} when $d = 1$).
    
    Under the restriction $\eta < 1/\LipK$, every map $\varphi_i$ is monotone on $I$ with respect to every orthont $\mathbb{R}_{\dtuple}^{\dimR}$.  This is because each component $\varphi_i^{(j)}$ of $\varphi_i$ (see \eqref{Eq:SepVarphi}) is an increasing function on $I^{(j)}$ via \cref{Prop:LeftRight}.
    %This follows directly from \eqref{Eq:SepVarphi} and \cref{Prop:LeftRight} which together show each component $\varphi_i^{(j)}$ of $\varphi_i$ is an increasing function on $I^{(j)}$.
    
    It remains to show $\{\varphi_i\}_{i=1}^n$ satisfy \eqref{Eq:BL_PathCond}.  For notational simplicity, we assume without loss of generality (e.g., after translation and scaling of $X_k$), that $T_{\Tindexd}$ has the form
    \begin{equation*}
        T_{\Tindexd}= [0,1]^\dimR \,.
    \end{equation*}
    We prove \eqref{Eq:BL_PathCond} by induction on the dimension $\dimR$.  
    
    Suppose $\dimR = 1$. By the definition of $T_{\Tindex}$ and \cref{Prop:FiniteTs}(b) we have $[0, 1) \subset R$ and $(0, 1] \subset L$. Thus, \eqref{Eq:LiminfL}--\eqref{Eq:LimsupR} in \cref{Prop:LeftRight} imply there exists paths $\vec{t}_1$ and $\vec{t}_2$ (which can be chosen to be the same length) that map $\varphi_{\vec{t}_1\,}(0) > x$ and $\varphi_{\vec{t}_2 \,}(1) < x$ where $ x$ is any point in the interior of $T_\Tindex = [0,1]$. Thus,  \eqref{Eq:BL_PathCond} holds with $\dtuple = +1$ and $x_0 = \frac{1}{2}$.
    
    Assume that \eqref{Eq:BL_PathCond} holds in $\mathbb{R}^{\dimR-1}$. That is, for the family of maps $\{\hat{\varphi}_i\}_{i=1}^{n}$ given as
    \begin{equation*}
        \hat{\varphi}_i(x_1, \, x_2, \, \ldots, \, x_{\dimR-1}) = 
        \begin{pmatrix}  \varphi_i^{(1)}(x_1), &  \varphi_i^{(2)}(x_2), & \hdots, & \varphi_i^{(\dimR-1)} (x_{\dimR-1})
        \end{pmatrix} \, ,
    \end{equation*} 
    there exists an $\hat{\dtuple} \in \{+1, -1\}^{\dimR - 1}$, $\hat{x}_0 \in \mathbb{R}^{\dimR - 1}$ and two paths $\vec{p}_1$ and $\vec{p}_2$ (each having the same length $\widehat{\ell}$) satisfying 
    \begin{equation}\label{Path_cond_Rd1}
        \hat{\varphi}_{\vec{p}_1} \big([0,1]^{\dimR-1} \big) \, \preceq_{\hat{\dtuple} } \hat{x}_0 
        \qquad  \textrm{and} \qquad  
        \hat{x}_0 \preceq_{\hat{\dtuple}} \hat{\varphi}_{\vec{p}_2}\big([0,1]^{\dimR-1} \big) \, .
    \end{equation}

    Next consider the case of dimension $\dimR$. For each map $\varphi_i$ write
    \begin{align*}
        \varphi_i(x) = \begin{pmatrix}
            \hat{\varphi}_i, & \varphi_i^{(\dimR)}(x_{\dimR})
        \end{pmatrix} \, , 
    \end{align*}
    where $\hat{\varphi}_i$ is the restriction of $\varphi_i$ to the first $\dimR-1$ coordinates. By hypothesis 
    $\{ \hat{\varphi_i}\}_{i=1}^{n}$ satisfies \eqref{Path_cond_Rd1} for some $\hat{\dtuple}$, $\hat{x}_0$ and paths $\vec{p}_1$ and $\vec{p}_2$.  We now construct $\dtuple \in \{+1, -1\}^{\dimR}$, $x_0 \in \mathbb{R}^{\dimR}$, and two paths for which \eqref{Eq:BL_PathCond} holds.

    Turning attention to the $\dimR$th coordinate, either 
    \begin{align*}
        \varphi_{\vec{p}_2}^{(\dimR)} (0) \geq \varphi_{\vec{p}_1}^{(\dimR)} (1) \qquad 
        \textrm{or} \qquad 
        \varphi_{\vec{p}_2}^{(\dimR)} (0) < \varphi_{\vec{p}_1}^{(\dimR)} (1) \, . 
    \end{align*}
    
     If $\varphi_{\vec{p}_2}^{(\dimR)} (0) \geq \varphi_{\vec{p}_1}^{(\dimR)} (1)$, then the $d$th coordinate satisfies the splitting condition in one dimension     
     \begin{align*}
         \varphi_{\vec{p}_1}^{(\dimR)}\big( [0,1] \big) \leq \Midpt \leq 
         \varphi_{\vec{p}_2}^{(\dimR)} \big( [0, 1] \big) \, , 
     \end{align*}
     where
     \begin{align}
        \Midpt := \frac{1}{2} \left( \varphi_{\vec{p}_2}^{(\dimR)}(0)  + \varphi_{\vec{p}_1}^{(\dimR)}(1) \right) \, .
     \end{align}
     Then \eqref{Eq:BL_PathCond} holds in dimension $\dimR$ for $\{\varphi_i\}_{i=1}^n$ with paths $\vec{p}_1$, $\vec{p}_2$, orthont $\dtuple = (\hat{\dtuple}, \, +1)$ and
     \begin{align}\label{Def:x0}
        x_0 := \begin{pmatrix}\hat{x}_0 \, , & \Midpt
        \end{pmatrix} \in \mathbb{R}^{\dimR} \, .
      \end{align}
    
    If $\varphi_{\vec{p}_2}^{(\dimR)} (0) < \varphi_{\vec{p}_1}^{(\dimR)} (1)$, then we construct two new paths (via concatenation) for which \eqref{Eq:BL_PathCond} holds.  Let $\LipK_0 :=  (1+\eta \LipK )^{\widehat{\ell}}$ and set
    \begin{equation*}
       \varepsilon :=\frac{1}{2 \LipK_0}\left( \varphi_{\vec{p}_1}^{(\dimR)} (1) -\varphi_{\vec{p}_2}^{(\dimR)} (0) \right) > 0 \, .
    \end{equation*}
    The choice of $\LipK_0$ with Assumption \ref{A4} then yields the following Lipschitz condition
    \begin{align}\label{Lip_path_Rd}
        \left| \varphi_{\vec{p}_j}^{(d)}(x) - \varphi_{\vec{p}_j}^{(d)}(y) \right| \leq \LipK_0 |x - y| \, , 
        \qquad j = 1,2 \quad \textrm{and} \quad x,y \in [0,1] \, . 
    \end{align}    

    Next, from \cref{Prop:LeftRight} there exist two paths $\vec{q}_1$ and $\vec{q}_2$ for which
   \begin{equation}
   \label{Small_xd_range1}
         \varphi_{\vec{q}_1}^{(d)}\big([0,1] \big) \subset   [1-\varepsilon,1]  \, ,
   \end{equation}
    and
    \begin{equation}\label{Small_xd_range2}
        \varphi_{\vec{q}_2}^{(d)}\big([0,1]\big) \subset [0,\varepsilon] \,.
    \end{equation}
    Applying $\varphi_{\vec{p}_1}^{(d)}$ to \eqref{Small_xd_range1} yields
    \begin{align}\label{Eq:SetInc1}
         \varphi_{\vec{p}_1 \circ \vec{q}_1}^{(d)}\big([0,1]\big) &\subset\varphi_{\vec{p}_1}^{(d)} \big( [1-\varepsilon, \, 1 ] \big) \, .  
    \end{align}    
    Using the fact that $\varphi_{\vec{p}_1}^{(d)}$ is increasing, together with \eqref{Lip_path_Rd} implies
    \begin{align}\label{Eq:SetInc2}
        \varphi_{\vec{p}_1}^{(d)}(1) - \varphi_{\vec{p}_1}^{(d)}(1- \varepsilon) 
        \leq K_0 \varepsilon \, . 
    \end{align}
    Substituting the definition of $\varepsilon$ into \eqref{Eq:SetInc2} yields $\varphi_{\vec{p}_1}^{(d)}(1-\varepsilon) \geq \Midpt$. Together with $\varphi_{\vec{p}_1}^{(d)}(1) \leq 1$, \eqref{Eq:SetInc1} implies
    \begin{align*}
        \varphi_{\vec{p}_1 \circ \vec{q}_1}^{(d)}\big([0,1]\big) 
         &\subset\left[ \Midpt , \, 1 \right] \, . 
         %&\subset\left[ \Midpt , \, \varphi_{\vec{p}_1}^{(d)}(1) \right] \, . 
    \end{align*}
    %   Here the last inclusion also uses the fact that $\varphi_{\vec{p}_1}^{(d)}$ is increasing.
    
    By a similar argument, applying the $\varphi_{\vec{p}_2}^{(d)}$ and \eqref{Lip_path_Rd} to \eqref{Small_xd_range2} yields
    \begin{align*}
         \varphi_{\vec{p}_2 \circ \vec{q}_2}^{(d)}\big([0,1]\big) \subset 
         \left[ 0, \Midpt \right] \, .
         %\left[ \varphi_{\vec{p}_2 }^{(d)}(0), \Midpt \right] \, .
    \end{align*}
    Hence, the $d$th coordinate satisfies the splitting condition 
    \begin{align*}
        {\varphi}_{\vec{p}_2 \circ \vec{q}_2}^{(d)}\big([0,1] \big) \leq \Midpt \leq {\varphi}_{\vec{p}_1 \circ \vec{q}_1}^{(d)}\big( [0,1] \big) \, .
    \end{align*}
    Subsequently, the maps $\{\varphi_i\}_{i=1}^n$ satisfy condition \eqref{Eq:BL_PathCond} in dimension $\dimR$ with paths $\vec{p}_1 \circ \vec{q}_1$ and $\vec{p}_2 \circ \vec{q}_2$, orthont $\dtuple = ( \hat{\dtuple}, \, -1)$ and point $x_0$ defined in \eqref{Def:x0}. 
    \end{proof}

%================================================================================
\subsection{Main Result Proof of Part \ref{Main_result_d}}
%================================================================================
The proof of part \ref{Main_result_d} utilizes the metric $\tilde{d}$ defined in \eqref{d_big_def} which satisfies the following: for any finite non-negative Borel measures $\mu_1, \mu_2, \nu_1, \nu_2$ supported on $I$, 
\begin{align}
    \label{d_a_tri}
     \tilde{d}(\mu_1+ \mu_2,\nu_1 + \nu_2) &\leq \tilde{d}(\mu_1,\nu_1) +\tilde{d}(\mu_2,\nu_2) \, ,  
 \\
\label{d_a_split}
     \tilde{d}(\mu_1+ \mu_2,\nu_1 + \nu_2) &\leq \nu_1(I) + \mu_1(I) + \tilde{d}(\mu_2,\nu_2).
\end{align}
Here \eqref{d_a_tri} and \eqref{d_a_split} follow from \eqref{dF_split11} and \eqref{d_F_tri} together with the fact that for $ \Tindexd \in \mathcal{\M}$, $T_{\Tindexd} \subset I$  and $B \subset I$ are pairwise disjoint.

\begin{proof}[Proof of Main Result \cref{Main_thm_basin}\ref{Main_result_d}]  {We first establish the properties of $\{ g_{\Tindexd}\}_{\Tindexd \in \mathcal{M}}$ and the integral representation of $c_{\Tindexd}$.}

Let $\mu_0 \in \mathscr{P}(I)$. Then for each $\Tindexd \in \mathcal{\M}$, $\mu_{k}( T_\Tindexd)$ is a bounded increasing sequence since $T_\Tindexd$ is positive invariant; see \cref{Main_thm_basin}\ref{Main_a}. {The sequence converges to a limit, call it
    \begin{equation}
     \label{c_mu_def}
          c_{\Tindexd}(\mu) := \lim_{k \rightarrow \infty} (\mathcal{P}^k \mu) (T_{\Tindexd}) \,.
     \end{equation} 
     Using this definition introduce the functions $\{g_{\Tindexd}\}_{\Tindexd \in \mathcal{M}}$ as
     \begin{align*}
         g_{\Tindexd}(x) := c_{\Tindexd}(\delta_x) \, \qquad ( x \in I) \,.
     \end{align*}}
    {From \eqref{c_mu_def}, the coefficients $c_{\Tindexd } \in [0,1]$. Furthermore, they satisfy $\sum_{\Tindexd \in \mathcal{\M} } c_{\Tindexd } = 1$ since $(\mathcal{P}^k \mu)(B) \rightarrow 0$ as $k \rightarrow \infty$ (see \cref{Main_thm_basin}\ref{Main_a}\ref{Main_a_iii}). Hence, the functions $g_{\Tindexd}$ are non-negative, $g_{\Tindexd}(x) = 1$ for $x \in T_{\Tindexd}$ and for each $x \in I$ satisfy $\sum_{\Tindexd \in \mathcal{M}} g_{\Tindexd}(x) = 1$. 

     The remaining properties of $g_{\Tindexd}$ will be shown with the use of an auxiliary function.  Fix $\Tindexd \in \mathcal{M}$ and let $\psi_0 : I \rightarrow [0,1]$ be any continuous function such that $\psi_0 \equiv 1$ on $T_{\Tindexd}$ and $\psi_0 \equiv 0$ on $T_{\Tindexd'}$ for all $\Tindexd' \neq \Tindexd$. Such a function exists since the $T_{\Tindexd}$ are disjoint rectangles.  It is helpful to view $\psi_0$ as a continuous version of a characteristic function for $T_{\Tindexd}$.    
     
     The properties of $\psi_0$, together with \cref{Main_thm_basin}\ref{Main_a}\ref{Main_a_iii}, imply
     \begin{align}\label{c_mu_rep1}
          c_{\Tindexd}(\mu) &= \lim_{k \rightarrow \infty} \int_I \psi_0(x) \, d{\mathcal{P}}^k \mu(x)  \\ \label{c_mu_rep2}
          &= \lim_{k \rightarrow \infty} \int_I \psi_k(x) \, d\mu(x)  \, .
     \end{align}
      Here \eqref{c_mu_rep2} follows from \eqref{c_mu_rep1} by introducing $\psi_k := {\mathcal{P}^\star}^k \psi_0$ $(k \geq 0)$ defined as in \eqref{Def:Pstar} and applying the adjoint property \eqref{Eq:AdjointProperty}.  Each $\psi_k$ is continuous since $\mathcal{P}^{\star}$ maps continuous functions on $I$ to continuous functions on $I$.  Note also that each $\psi_k$ may be written as
      %Moreover, the adjoint property gives %Using the fact that $\mathcal{P}$ is the adjoint of $\mathcal{P}^{\star}$      
    \begin{align}\label{Eq:DualityOfPsi}        
        \psi_k(x) = \langle {\mathcal{P}^\star}^k \psi_0, \delta_x\rangle = \langle \psi_0, \mathcal{P}^k \delta_x\rangle \,.
    \end{align}       
    We now claim that $\psi_k$ converges uniformly to $g_{\Tindexd}$ on $I$: For any $x \in I$ 
      \begin{alignat*}{2}
        \left |g_\Tindexd(x) - \psi_k(x) \right| &= \left |c_{\Tindexd}(\delta_x) - \int_I \psi_0(y)  \, d\mathcal{P}^k \delta_x(y)  \right| && \qquad
        (\textrm{Using \, \eqref{Eq:DualityOfPsi}}) \\
        & \leq \left |c_{\Tindexd}(\delta_x) -  (\mathcal{P}^k \delta_x)(T_\Tindexd)  \right| + (\mathcal{P}^k \delta_x)(B) &&\qquad
        (\textrm{By definition of } \psi_0) \\
        & \leq 2  \left( 1 - \frac{1}{n^{\ell_0}} \right)^{\lfloor \frac{k}{\ell_0} \rfloor} \, &&\qquad (\textrm{From \cref{Main_thm_basin} \ref{Main_a}\ref{Main_a_iii}}). 
     \end{alignat*}     
     Thus, $g_\Tindexd$ is continuous since it is the uniform limit of continuous functions and satisfies \newline $\int_{I} g_{\Tindexd} \, d\mu_{\Tindexd'}^{\star} = \delta_{\Tindexd\Tindexd'}$.  Applying the dominated convergence theorem to \eqref{c_mu_rep2} yields
     \begin{equation*}
           c_{\Tindexd}(\mu) %= \lim_{k\rightarrow \infty} \int_I  \psi_k(x) \, d\mu(x) \, 
           =\int_I  g_\Tindexd (x) \, d\mu(x) \, .
     \end{equation*}
     Next, $g_{\Tindexd}$ is an eigenfunction satisfying $\mathcal{P}^{\star} g_{\Tindexd} = g_{\Tindexd}$.  This follows since \eqref{Def:Pstar} ensures $\|\mathcal{P}^{\star} f\|_{\infty} \leq \|f\|_{\infty}$ for any continuous function $f$ on $I$, so that one has
     \begin{alignat*}{2}
         \|\mathcal{P}^{\star} g_{\Tindexd} - g_{\Tindexd} \|_{\infty} &\leq 
         \|\mathcal{P}^{\star} g_{\Tindexd} - \mathcal{P}^{\star}\psi_k  \|_{\infty} + \|\psi_{k+1} - g_{\Tindexd}\|_{\infty}  && \\
         &\leq 
         \|g_{\Tindexd} - \psi_k  \|_{\infty} + \|\psi_{k+1} - g_{\Tindexd}\|_{\infty}  \\
         &\rightarrow 0  \qquad 
          \textrm{as } \qquad k \rightarrow \infty.&& 
     \end{alignat*}     
     We have shown that $g_{\Tindexd}$ satisfies the properties \cref{Main_thm_basin}\ref{Main_result_d}\ref{Main_d_i}--\ref{Main_d_iii}. It is also the unique function satisfying these properties. For if $h_{\Tindexd}$ were a second such function, set $\psi_0 = h_{\Tindexd}$ above. But then $\psi_k = h_{\Tindexd}$ for all $k$ since $h_{\Tindexd}$ is an eigenvector of $\mathcal{P}^{\star}$ and so $h_{\Tindexd} = g_{\Tindexd}$. %$\mathcal{P}^{\star} h_{\Tindexd} = h_{\Tindexd}$
     }
   
    Next introduce the probability measure
    \begin{equation*}
         \mu^\star := \sum_{\Tindexd  \in \mathcal{\M}}  c_{\Tindexd } \mu_{\Tindexd }^\star \, 
    \end{equation*}
    {where, for brevity, we suppress the dependence on $\mu$, i.e., $c_{\Tindexd} = c_{\Tindexd}(\mu)$}.  Note that $\mu^\star$ is the convex combination of invariant measures and hence is invariant.

       We now show the result
    \begin{equation}\label{Eq:MainEstimate} 
        \Tilde{d} \left ( {\mathcal{P}}^{2 \ell k} \mu_0, \mu^\star \right) \leq 3 \gamma^k \, .
    \end{equation}
    where
    \begin{align*}        
        \ell := \ell_0 \vee \max_{ \gamma_{\Tindexd}  \in \mathcal{\M}} \{ \ell_{\Tindexd} \} \qquad 
        \textrm{and} \qquad
        \gamma := 1 - \frac{1}{n^{\ell}} \, ,
    \end{align*}
    with $\plen_0$ and $\ell_{\Tindexd}$ defined in \cref{Main_thm_basin}\ref{Main_a}-\ref{Main_b}.
    
    For notational brevity, also introduce
    \begin{align*}        
        \tilde{\mu}_{k} := \mu_{ \ell k} \qquad \textrm{and} \qquad 
        \mathcal{Q} := {\mathcal{P}}^{\ell}  \, .
    \end{align*}
    
    With these notations, and using the fact that ${\mathcal{Q}}^{2k} \mu_0 = {\mathcal{Q} }^{k} \tilde{\mu}_k$ we can decompose the measure
    \begin{align}\label{Eq:DecompQ2k}
        {\mathcal{Q}}^{2k} \mu_0 = 
         {\mathcal{Q}}^k \left( \tilde{\mu}_{k} \big |_{B} \right) + 
         \sum_{ \Tindexd \in \mathcal{\M}} {\mathcal{Q}}^k \left( \tilde{\mu}_{k} \big |_{T_\Tindexd } \right) \, ,
    \end{align}
    into its components in $B$ and separately in $T_\Tindexd $. 
    We can also write
    \begin{align}\label{Eq:DecompInvk}
         \mu^\star = 
         \sum_{\Tindexd  \in \mathcal{\M} }  \big(c_\Tindexd -\tilde{\mu}_{k}(T_\Tindexd ) \big) \, \mu_{\Tindexd}^\star+    \sum_{\Tindexd \in \mathcal{\M} } \tilde{\mu}_{k}(T_\Tindexd) \, \mu_{\Tindexd}^\star \, ,
    \end{align}
    measuring the discrepancy between the mass on $T_\Tindexd$ and $c_\Tindexd$.

    Substituting \eqref{Eq:DecompQ2k} and \eqref{Eq:DecompInvk} into $\tilde{d}$ and using the inequality \eqref{d_a_split} yields:
      \begin{alignat*}{2}
      \tilde{d} \left(  {\mathcal{Q}}^{2k}  \mu_0, \mu^\star \right) &\leq  
        \mathcal{I}_1 + \mathcal{I}_2 + \mathcal{I}_3 \, , 
    \end{alignat*}
    where
    \begin{align*}
        \mathcal{I}_1 &= {\mathcal{Q}}^{k}  \left( \tilde{\mu}_{k} \big |_{B} \right) (I) \, ,
        \phantom{\left(     \sum_{\Tindexd \in \mathcal{\M}} \right)}
   \\
        \mathcal{I}_2 &=    \sum_{\Tindexd \in \mathcal{\M} }  \big(c_\Tindexd-\tilde{\mu}_{k}(T_\Tindexd) \big) \, \mu_{\Tindexd}^\star (I)  \, , \phantom{\left(     \sum_{\Tindexd \in \mathcal{\M}} \right)}
   \\
        \mathcal{I}_3 &= \tilde{d} \left(     \sum_{\Tindexd \in \mathcal{\M}} {\mathcal{Q}}^k \left( \tilde{\mu}_{k} \big |_{T_\Tindexd} \right) , \sum_{\Tindexd \in \mathcal{\M} } \tilde{\mu}_{k}(T_\Tindexd) \, \mu_{\Tindexd}^\star  \right ) \, .
    \end{align*}
    The result \eqref{Eq:MainEstimate} is then proved provided $\mathcal{I}_i \leq \gamma^{k}$ for each $i = 1,2,3$.

     The first two terms $\mathcal{I}_1, \mathcal{I}_2$ measure the mass of $\mu_0$ that remains in the uniformly transient region $B$ after $k$ iterations of $\mathcal{Q}$. We have, 
    \begin{alignat}{2}\label{Eq:EstI1}
        \mathcal{I}_1 &= \tilde{\mu}_{k} \big |_{B}(I) =
        \tilde{\mu}_{k}(B) \leq \gamma^{k} \, . 
    \end{alignat}
        The first equality in \eqref{Eq:EstI1} follows since $I$ is positive invariant and $\mathcal{Q}$ is a Markov operator; the last inequality follows from \eqref{Eq:Mass_On_B}. 
        
    Similarly, since $\mu_{\Tindexd}^{\star}(I) = 1$ and the $c_\Tindexd$'s sum to one, the second term is 
    \begin{align*}
        \mathcal{I}_2 &= \sum_{\Tindexd \in \mathcal{\M} }  \big(c_\Tindexd-\tilde{\mu}_{k}(T_\Tindexd) \big)  = 1 - \tilde{\mu}_k(T) = \tilde{\mu}_k(B) \leq \gamma^k \, .
    \end{align*}
 Estimating the third term $\mathcal{I}_3$ follows from \cref{Main_thm_basin}\ref{Main_b} which quantifies the convergence of the (normalized) restriction $\tilde{\mu}_k\big|_{T_\Tindexd}$ to $\mu_{\Tindexd}^{\star}$.  In particular, for any non-negative measure $\nu$ supported on $T_{\Tindexd}$ from \cref{Main_thm_basin}\ref{Main_b} we obtain for all $k \geq 0$ and $ \Tindexd \in \mathcal{\M}$:
  \begin{alignat}{2}\label{Con_TiNu}    
        \tilde{d} \big( {\mathcal{Q} }^{k} \nu, \, 
        | \nu | \, \mu_{\Tindexd}^\star\big) &=  
        d_{\alpha_{\Tindexd}}\Big( {\mathcal{P} }^{k \ell_{\Tindexd} + k (\ell- \ell_{\Tindexd}) } \nu, \,  |\nu| \, \mu_{\Tindexd}^\star\Big)  && \\ \nonumber
        &\leq  |\nu| \, \gamma_{\Tindexd}^k \, 
        d_{\alpha_{\Tindexd}} \Big( {\mathcal{P} }^{  k (\ell- \ell_{\Tindexd}) } \frac{\nu}{|\nu|}, \, \mu_{\Tindexd}^\star\Big) && (\textrm{By} \; \eqref{dF_split1} \textrm{ and } \eqref{Con_Tr_Rd} ) \\ \nonumber 
        &\leq  | \nu| \,  \gamma^k \,.   &\qquad&   
    \end{alignat}
    Applying \eqref{Con_TiNu} to $\mathcal{I}_3$ gives
     \begin{alignat*}{2}
        \mathcal{I}_3 &\leq \sum_{\Tindexd \in \mathcal{\M}}  \tilde{d} \left(  {\mathcal{Q}}^k \left( \tilde{\mu}_{k} \big |_{T_\Tindexd} \right) , \, \tilde{\mu}_{k}(T_\Tindexd) \, \mu_{\Tindexd}^\star   \right ) 
        &\qquad & (\textrm{By } \eqref{d_a_tri}) \, , \\
        &\leq \sum_{\Tindexd \in \mathcal{\M} }  \tilde{\mu}_{k}(T_{\Tindexd})  \, \gamma^k 
        && (\textrm{By} \; \eqref{Con_TiNu} \; \textrm{with} \; \nu = \tilde{\mu}_k\big|_{T_\Tindexd}) \, ,\\ %\; 
        &= \tilde{\mu}_{k}(T) \, \gamma^k \leq \gamma^{k} \, . && 
        (\textrm{Since} \; \tilde{\mu}_k(T) \leq 1) \, ,
    \end{alignat*}
    which proves \eqref{Eq:MainEstimate} and hence \eqref{Eq:MainCvgMetric}.
    
    Lastly, in the case where $d=1$ and $I=[a,b]$, note that for any two probability measures $\mu$ and $\nu$,
    \begin{align*}
        d_F(\mu,\nu ) &= \sup_{x \in [a,b]} |\mu([a,x]) -\nu([a,x]) | \\
        &\leq  \sup_{x \in [a,b]}  |\mu([a,x] \cap B ) -\nu([a,x] \cap B ) | + \sum_{\Tindexd \in \mathcal{\M}}  |\mu([a,x] \cap T_{\Tindexd}) -\nu([a,x] \cap T_{\Tindexd})|
        \\
        & \leq \tilde{d}(\mu,\nu) \, ,
    \end{align*}
    since $B \cup \left(\cup_{\Tindexd \in \mathcal{\M}} T_{\Tindexd} \right) =I$ and are disjoint.
\end{proof}

{We conclude with a proof of the Corollary.}
\begin{proof}[Proof of \cref{Coro_unique}] 
    {%From \cref{Main_thm_basin} it is sufficient 
    It is sufficient to show that for each dimension $j \in [\dimR]$ the collection of functions $\{f_i^{(j)}\}_{i=1}^n$ admits exactly one interval $T_1^{(j)}$ satisfying \cref{Def:SetsTj}. In this case there is a single set $T_\Tindexd$ and a unique invariant measure.

    Consider $d=1$, with functions $\{f_i\}_{i=1}^n$ on $\mathbb{R}$ and suppose $f_{i^{\star}}$ (for some $i^{\star}$) has a single critical point at $x_0$. From \ref{A3} we then have that 
    \begin{equation}
    \label{Cor_equ}
        (x_0, \infty) \subset L  \qquad \textrm{and} \qquad (-\infty,x_0) \subset R \, .
    \end{equation}
    We claim that every $T_m =[l,r]$ satisfying \cref{Def:SetsTj} must contain $x_0$. Since if $x_0 \notin [l,r]$, then either $l > x_0$ or $r <x_0 $ and \eqref{Cor_equ} implies that either 
      \begin{equation*}
       [l,r] \subset L  \qquad \textrm{or} \qquad [l,r] \subset R \, ,
    \end{equation*}
    which contradicts $l \in \partial L$ and $r \in \partial R$. Thus, $x_0 \in T_m$ and \cref{Prop:FiniteTs}(d) implies that $T_m$ is unique.}     
\end{proof}
%=====================================
 \section{Conclusions and outlook}  
{
In this work we analyze SGD for a class of separable objective functions and  establish a ``Doeblin-type decomposition'' of the state space into a disjoint union of closed absorbing rectangles and a single  uniformly transient set. We show that for each absorbing rectangle $T_{\Tindexd}$ there is a unique invariant measure $\mu_{\Tindexd}$ whose support is contained in $T_{\Tindexd}$, and that the set of all invariant measures is the convex hull of the $\mu_{\Tindexd}$'s. %In addition, if SGD is initialized inside a absorbing rectangle $T_{\Tindexd}$, then the dynamics converge at a geometric rate to the corresponding measure $\mu_{\Tindexd}$.  
Lastly, we show that for arbitrary initializations, SGD converges at a geometric rate to a convex combination of the  $\mu_{\Tindexd}$'s.  Despite the fact that the long time dynamics may be complicated by multiple invariant measures, geometric convergence rules out the possibility of more complicated dynamics for the $\mu_k$ such as periodic orbits.  

One of the motivations for this work is to provide a greater understanding of the constant step-size SGD dynamics as a model for the early stages of  vanishing step-size SGD. For separable functions, \cref{Main_thm_basin} provides the following picture. If one runs SGD in two phases: a first phase with constant steps sizes (long enough to sample the invariant measure $\mu^{\star}$), followed by a second phase with a sufficiently fast vanishing learning rate $\eta_k \rightarrow 0$, the dynamics will effectively perform gradient descent with initial conditions provided by $\mu^{\star}$. The probability of converging to a minimum $x^{\star}$ of $F$ is the mass $\mu^{\star}(A)$ where $A$ is the basin of attraction of $x^{\star}$. 

The fact that SGD may have multiple invariant measures implies that SGD, even for large times, will not in general explore the entire state space.  This may be of concern if one seeks to use SGD to perform global optimization.  For instance, example \ref{Subsec:ExGlobalMinFail} highlights the failure of SGD to perform global optimization: regardless of initialization or $\eta$, the probability laws $\mu_k$ (as $k \rightarrow \infty$) will contain zero mass in the neighborhood of the global minima.  All is not lost, however, as some structure on $f_i$ can be imposed to ensure a unique invariant measure (\cref{Coro_unique} provides a simple condition). In this case, the long time behavior is independent of initialization. 

Another aspect of this work is a cautionary tale of how the diffusion approximation may fail to predict the long-term behavior of SGD. In the diffusion approximation, there is always a small but non-zero probability of transitioning between local minima, and independent of initialization, the dynamics  $\rho(x,t)$  eventually converge to a unique invariant measure. However, as highlighted in examples \ref{Subsec:DoubleWell} and \ref{Subsec:ExGlobalMinFail}, SGD iterates may get trapped in local minima, leading to multiple invariant measures. Thus, in general, the large-time dynamics $\mu_k$  (as $k \rightarrow \infty$) of SGD depend on the initialization $\mu_0$ of SGD.

Lastly, this work serves as a building block for additional work on SGD dynamics.  For instance, \cref{Main_thm_basin} provides a stepping stone to establish a \emph{transition state} theory for SGD (e.g., \cite{AzizianIutzelerMalickMertikopoulos2024}) which can provides additional asymptotics (in $\eta$) on the mass of the invariant measure in the vicinity of local minima of $F$. Another natural question is whether the main result here \cref{Main_thm_basin} generalizes to the case of non-separable functions? An affirmative answer will rule out the possibility of other complex dynamics such as the probability laws $\mu_k$ traversing periodic orbits.  Just as the concept of \emph{basins of attraction} provides a useful theoretic (albeit often computationally intractable) framework for understanding gradient descent dynamics, absorbing sets could provide a similar conceptual theory for understanding the long-time behavior of SGD. }

%================================================================================    
\section*{Acknowledgments}
%================================================================================
The authors would like to thank William Joseph McCann for helpful discussions, and Sean Lawley for highlighting the connection between Infinite Bernoulli convolutions and SGD. {We would also like to thank the anonymous reviewers for comments which led to improvements in the paper.}

This material is based upon work partially supported by the National Science Foundation under Grants No.\ DMS--2012268 and DMS--2309727.  Any opinions, findings, and conclusions or recommendations expressed in this material are those of the authors and do not necessarily reflect the views of the National Science Foundation.

%\nocite{*}
\bibliographystyle{siamplain}
\bibliography{refs}

\appendix

\end{document}